\pdfoutput=1
\documentclass[a4paper]{article}
\usepackage{geometry}
\usepackage{amsthm}
\usepackage[numbers]{natbib}
\usepackage{amsmath,amssymb,amsfonts}
\usepackage{algorithm}
\usepackage{algpseudocode}
\usepackage{booktabs}
\usepackage{subcaption}
\usepackage{enumitem}
\usepackage[european, straightvoltages, straightlabels]{circuitikz}
\usepackage{pgfplots}
\usepackage{pgfplotstable}
\pgfplotsset{compat=1.18}
\usepackage{graphicx}
\usepackage{bbm}
\usepackage{xcolor}
\usepackage{hyperref}

\theoremstyle{plain}
\newtheorem{theorem}{Theorem}[section]
\newtheorem{lemma}[theorem]{Lemma}
\newtheorem{proposition}[theorem]{Proposition}

\theoremstyle{definition}
\newtheorem{definition}[theorem]{Definition}
\newtheorem{assumption}[theorem]{Assumption}
\theoremstyle{remark}
\newtheorem{remark}[theorem]{Remark}

\newcommand{\diff}{\mathrm{d}}
\newcommand{\ddt}{\frac{\diff}{\diff t}}
\newcommand{\R}{\mathbbm{R}}

\newcommand{\rnn}{\mathbbm{R}^{n \times n}}
\newcommand{\gHam}{\mathcal{H}_\mathrm{DAE}}
\newcommand{\p}{\varphi}
\newcommand{\bI}{\mathcal{I}_i}
\newcommand{\kIi}{I_i}

\definecolor{darkgreen}{rgb}{0.1,0.7,0.1}
\graphicspath{{./figures/}}

\begin{document}

\title{Goal-Oriented Time Adaptivity for Linear
  Port-Hamiltonian Differential-Algebraic Equations of Index~1%
  \thanks{This work was funded by the Deutsche Forschungsgemeinschaft
    (DFG, German Research Foundation) -- Project-ID 531152215 -- CRC 1701.}}

\author{Aashutosh Sharma\thanks{Department of Applied and Computational Mathematics,
  University of Wuppertal, Gau{\ss}stra{\ss}e 20, 42119 Wuppertal, Germany.
  \texttt{asharma@uni-wuppertal.de}}
  \and
  Andreas Bartel\thanks{Department of Applied and Computational Mathematics,
  University of Wuppertal, Gau{\ss}stra{\ss}e 20, 42119 Wuppertal, Germany.
  \texttt{bartel@uni-wuppertal.de}}
  \and
  Manuel Schaller\thanks{Faculty of Mathematics, Chemnitz University of Technology,
  Stra{\ss}e der Nationen 62, 09111 Chemnitz, Germany.
  \texttt{manuel.schaller@math.tu-chemnitz.de}}}

\date{}

\maketitle

\begin{abstract}
Port-Hamiltonian systems provide a highly-structured framework for modeling
of physical systems. By definition, they encode a balance equation relating
energy changes to supplied and dissipated energy. Capturing this energy
balance in discrete approximations is a fundamental challenge and often has
been achieved by designing particular schemes such as discrete gradient
methods. In this work, we propose an approach that controls the energy
balance violation for port-Hamiltonian differential algebraic equations via
time adaptivity using a posteriori grid refinement techniques based on the
dual weighted residual method. In particular, we show how one may leverage
the port-Hamiltonian structure to efficiently compute the error estimators
using a dissipativity-exploiting block-Jacobi approximation. We illustrate
the efficacy of the method by means of simulations of electrical circuit
models.


\end{abstract}
\section{Introduction}\label{sec:intro}

Port-Hamiltonian (pH) systems provide a systematic framework for the
energy-based modeling of coupled systems including networks and multiphysics applications.   
By encoding energy
storage, dissipation, and interconnection directly on the structural
level, the pH formulation guarantees fundamental physical properties
such as passivity and energy conservation by
construction~\cite{VanDerSchaftJeltsema2014,jacob2012linear}.  Central to any pH model
is its \emph{energy balance}, which relates the rate of change of the
Hamiltonian to externally supplied and internally dissipated power.
For simulations, in particular over long time horizons, preserving or accurately controlling this
balance at the discrete level is essential.

Many applications, including electrical networks~\cite{Guenther2024}, constrained multi-body
systems~\cite{Berger2025}, and spatially discretized field theories~\cite{Clemens2024}, involve both dynamic
and static coupling relations and are naturally formulated as
(port-Hamiltonian) differential-algebraic equations
(DAEs)~\cite{BeattieMehrmannXu2018,mehrmann2023control}.  In the work at hand, we consider linear
port-Hamiltonian DAEs of the form
\begin{subequations}\label{eq:ph-dae}

\begin{align}
  \ddt E\, x(t) &= (J - R)\, Q\, x(t) + B\, u(t),
    \label{eq:ph-dae-state}\\
  y(t) &= B^\top Q\, x(t),
    \label{eq:ph-dae-output}
\end{align}
\end{subequations}
where $E \in \R^{n\times n}$ is a possibly singular, 
$J = -J^\top\in \R^{n\times n}$ encodes energy-conserving interconnection,
$0\le R =R^\top \in \R^{n\times n}$ models dissipation, $Q \in \R^{n\times n}$ defines the energy metric and satisfies 
$E^\top Q = Q^\top E$, and $B\in \R^{n\times m}$ is
the port matrix. In \eqref{eq:ph-dae}, $x(t)\in \R^{n}$ is the descriptor state, $u(t)\in \R^{m}$ denotes the input and $y(t)\in \R^{m}$ is the \emph{conjugated output}. 

From these properties, it is easily seen that the quadratic Hamiltonian 
$\gHam(x) = \tfrac12\, x^\top E^\top Q\, x$ satisfies
the \emph{power balance}
\begin{equation}\label{eq:power-balance}
  \ddt \gHam(x(t))
  = y(t)^\top u(t) - (Qx(t))^\top R\, Qx(t),
\end{equation}
which yields, upon integration over $[s,t]\subset \R$, the \emph{energy balance}
\begin{equation}\label{eq:energy-balance}
  \gHam(x(t)) - \gHam(x(s))
  = \int_s^t \!\bigl[y(\tau)^\top u(\tau)
    - (Qx(\tau))^\top R\, Qx(\tau)\bigr]\,\diff\tau.
\end{equation}

In time discretization of the dynamics \eqref{eq:ph-dae}, it is thus clearly desirable to capture this central property of the dynamics. 
There are various existing approaches to tackle this challenge.
Structure-preserving time integrators enforce a discrete analogue
of~\eqref{eq:energy-balance} exactly.  Discrete gradient methods,
originally introduced for Hamiltonian
ordinary differential equations (ODEs) in~\cite{Gonz96}, have been extended to port-Hamiltonian
ODEs~\cite{KotyLefe19, schulze2023structure} and, more recently, to
pH-DAEs~\cite{kinon2025discrete, mehrmann2019structure,
morandin2024modeling}.  Generally speaking, these methods enforce an energy balance by design and employment of particular integrators. 

In this work, we pursue a complementary approach using \emph{goal-oriented grid adaptivity} that ensures an \emph{approximate}
energy balance by adaptation of the time grid, with full freedom in the choice of the time integration method.
Rather than enforcing
energy conservation exactly, we treat violations
of~\eqref{eq:energy-balance} as a computable \emph{quantity of interest}

and adapt the time grid to reduce these violations to a prescribed tolerance. The theoretical basis is the 
\emph{dual weighted residual} 
method, which we summarize next. 

Classical a~posteriori error estimators for time-stepping schemes 
control a global norm of the discretization error, such as
$\|x - x_k\|_{L^2(0,T)}$, and refine wherever this norm is 
locally large. While such estimators are well-established and reliable, they do not distinguish between error components that 
affect a given quantity of interest and those that do not. In many 
applications, including the energy-balance monitoring considered  here, the relevant quantity depends on the solution in a highly 
nonuniform way, and global-norm control leads to unnecessary 
refinement in regions that contribute negligibly to the error in the 
quantity of interest.

The Dual Weighted Residual (DWR) method~\cite{BeckerRannacher2001, 
Rannacher2003} addresses this by formulating error estimation 
directly in terms of a user-specified quantity of interest
$\mathcal{J}$. The method proceeds in three conceptual steps.
First, one defines a \emph{residual} $\rho(x_k)$ that measures how 
well the discrete solution $x_k$ satisfies the continuous equation; 
for a consistent discretization, this residual vanishes when tested 
against discrete test functions and is nonzero only when tested 
against functions outside the discrete space.
Second, one introduces an \emph{adjoint problem} whose right-hand 
side is the derivative $\mathcal{J}'(x_k)$ of the quantity of 
interest. The adjoint solution $z$ encodes the sensitivity of 
$\mathcal{J}$ to perturbations at each point in time: it is large 
where local errors propagate strongly into the quantity of interest 
and small where they do not.
Third, one derives an \emph{error representation}
\begin{equation}\label{eq:DWR-schematic}
  \mathcal{J}(x) - \mathcal{J}(x_k) 
  \approx \sum_{i=1}^{N} \eta_i, \qquad
  \eta_i = \rho_i(x_k)(z - z_k),
\end{equation}
where $\rho_i$ denotes the restriction of the residual to the $i$-th 
time interval and $z_k$ is a discrete approximation of the adjoint. 
Each indicator $\eta_i$ is the product of the local residual and the 
local adjoint error, so it is large only when both the discretization 
error and the sensitivity are simultaneously significant. This 
multiplicative structure is the key advantage of goal-oriented 
estimation: it concentrates refinement precisely where it most 
efficiently reduces the error in~$\mathcal{J}$.

The identity~\eqref{eq:DWR-schematic} is made rigorous in 
Section~\ref{sec:DWR-analysis}, where we specialize it to the dG(0) 
discretization of the pH-DAE and derive the computable local 
indicators used in the adaptive algorithm.

For port-Hamiltonian ODEs, goal-oriented time adaptivity targeting 
the energy balance~\eqref{eq:energy-balance} was recently developed 
in~\cite{BartelSchaller_2025}. In the present work, we extend this 
approach to DAEs of index one, which introduces several additional 
challenges: the algebraic constraints couple differential and 
algebraic variables with different regularity; the Schur complement 
reduction to the differential variable produces a generally 
non-symmetric governing operator; and the adjoint problem features 
jump conditions at the time nodes arising from the quartic structure 
of the energy-balance QoI. For a comprehensive treatment of the DWR 
framework, we refer to~\cite{Rannacher2003, BeckerRannacher2001}; 
applications to parabolic optimal control problems are developed 
in~\cite{meidner2008}; and two-sided effectivity bounds are analyzed 
in~\cite{EndtmayerLangerWick2020}.

\paragraph{Contributions.} %
The main novelties of this work are:
\begin{enumerate}[label=(\roman*)]
\item Definition of a tailored goal-oriented time adaptive scheme for pH-DAEs using a 
Schur-complement
  elimination of algebraic variables.
\item A DWR-based a~posteriori error representation of the energy balance error and derivation of computable
  local error indicators.
\item A parallelizable Block-Jacobi approximation of the backward
  adjoint problem exploiting dissipativity of the pH structure.
\item Numerical validation demonstrating significant efficiency gains
  over uniform time stepping.
\end{enumerate}

The codes for all the experiments conducted in this paper can be found at 
\begin{equation*}
    \text{https://github.com/cacheFriendlySharma/DWR-Linear-pHDAE-Index-1
}
\end{equation*}

\paragraph{Outline.}
The organization of this work is as follows. Section~\ref{sec:weak-formulation} introduces the pH-DAE model and weak
formulation.
Section~\ref{sec:discretization} presents the temporal discretization
and a~priori analysis.
A~posteriori error estimation is derived in
Section~\ref{sec:DWR-analysis}.
Efficient implementation and adjoint computation are discussed in
Section~\ref{sec:efficient_adjoint} and the parallel adjoint approximation is proposed in
Section~\ref{sec:efficient_adjoint}. Numerical results are contained in
Section~\ref{sec:numerical_experiments}, followed by conclusions.


\section{Port-Hamiltonian DAEs in weak form}
\label{sec:weak-formulation}

In this section, we define the pH-DAEs of index-1 in a semi-explicit formulation. Then, a variational formulation is introduced for preparation of the goal-oriented scheme.
Subsequently, we define the energy-balance-based goal functional and derive the corresponding
adjoint problem.

\subsection{Problem formulation and semi-explicit structure}
\label{sec:formulation}

Throughout this work, we impose the following structural assumptions on
the pH-DAE~\eqref{eq:ph-dae}.

\begin{assumption}\label{ass:ph-structure}
The system matrices in~\eqref{eq:ph-dae} satisfy:
\begin{enumerate}[label=(\alph*)]
  \item $E \in \rnn$ is singular with $\mathrm{rk}(E) = r < n$.
  \item $Q \in \rnn$ is non-singular  and $E^\top Q = Q^\top E > 0$ (symmetric positive definite).
  \item $J = -J^\top$ (skew-symmetric) and
        $R = R^\top \ge 0$ (symmetric positive semi-definite).
  \item $B \in \R^{n \times m}$, and the input satisfies
        $u \in L^2(0,T;\R^m)$.
  \item The matrix pencil $(E,\, (J-R)Q)$ is regular , i.e., $\operatorname{det}(sE-(J-R)Q)$ is not the zero polynomial.

  \item The DAE~\eqref{eq:ph-dae} is of
\emph{index-$1$}, i.e., the mapping
$x \mapsto (J - R)Q\,x$ to the kernel of~$E$ is invertible.
 \item $x_0\in \mathbb{R}^n$ is a consistent initial value, i.e., there is a solution to 
 \eqref{eq:ph-dae} that satisfies $x(0)=x_0$.
\end{enumerate}
\end{assumption}

For more details on the port-Hamiltonian setting, we refer to~\cite{BeattieMehrmannXu2018}.

To make the index-1 structure more  
explicit, we employ a \emph{kernel splitting
pair}~\cite{Jansen_2015}. To this end, 
let $V \in \R^{n \times r}$ and
$W \in \R^{n \times (n-r)}$ be such that $EW = 0$, the matrix
$(V,\, W) \in \rnn$ is non-singular, and
\begin{equation*}
  x = V x_1 + W x_2.
\end{equation*}
Left-multiplying the pH-DAE~\eqref{eq:ph-dae-state} by
$(V,\,W)^\top Q^\top$ produces the semi-explicit form
\begin{equation}\label{eq:Semi-Explicit form}
  \ddt
  \begin{pmatrix} E_{11} & 0 \\ 0 & 0 \end{pmatrix}
  \begin{pmatrix} x_1(t) \\ x_2(t) \end{pmatrix}
  =
  \begin{pmatrix} A_{11} & A_{12} \\ A_{21} & A_{22} \end{pmatrix}
  \begin{pmatrix} x_1(t) \\ x_2(t) \end{pmatrix}
  +
  \begin{pmatrix} B_1 \\ B_2 \end{pmatrix} u(t),
\end{equation}
with block matrices
\begin{equation*}
  (V,W)^\top Q^\top E\,(V,W)
  = \begin{pmatrix} E_{11} & 0 \\ 0 & 0 \end{pmatrix},
  \quad
  A = (V,W)^\top Q^\top (J - R)\, Q\,(V,W),
  \quad
  \begin{pmatrix} B_1 \\ B_2 \end{pmatrix}
  = \begin{pmatrix} V^\top Q^\top B \\ W^\top Q^\top B \end{pmatrix}.
\end{equation*}
The pH structure is inherited: 
$E_{11}^\top = E_{11} > 0$ 
and the governing matrix is dissipative in the Euclidean inner product, that is,
$A + A^\top \le  0$.  The index-$1$ condition ensures
that $A_{22}$ is invertible, so the algebraic variable can be determined
from the differential variable and the input at every time instant. For a discussion of transformations of pH-DAEs with a more general structure 
we refer again to~\cite{BeattieMehrmannXu2018}. Linear-algebraic properties of pH-matrix pencils (such as the index) were analyzed in~\cite{mehl2018linear, mehl2022matrix, faulwasser2022optimal}.


\paragraph{Underlying primal ODE problem.}

With projected initial value $x_{1,0}\in\R^{r}$, \eqref{eq:Semi-Explicit form} yields:
\begin{subequations}\label{eq:semi-explicit-expanded}
\begin{align}
  E_{11}\, \ddt x_1(t)
    &= A_{11}\, x_1(t) + A_{12}\, x_2(t) + B_1\, u(t), \qquad x_1(0) = x_{1,0} 
    \label{eq:Semi-Explicit-First-Eq} \\
  0 &= A_{21}\, x_1(t) + A_{22}\, x_2(t) + B_2\, u(t)
    \label{eq:Semi-Explicit-Second-Eq}
\end{align}
\end{subequations}

Since $A_{22}$ is invertible, we solve~\eqref{eq:Semi-Explicit-Second-Eq}
for $x_2$ and define the \emph{algebraic reconstruction}
\begin{equation}\label{eq:algebraic-reconstruction}
  \hat x_2(x_1, t)
    := -A_{22}^{-1}\bigl(A_{21}\, x_1 + B_2\, u(t)\bigr).
\end{equation}
Substituting into~\eqref{eq:Semi-Explicit-First-Eq} yields the
\emph{reduced evolution equation}
\begin{equation}\label{eq:Reduced-System}
  E_{11}\, \ddt x_1(t) + S\, x_1(t) = F(t),
\end{equation}
with the Schur complement operator and reduced forcing
\begin{equation}\label{eq:S-and-F-def}
  S := -\bigl(A_{11} - A_{12}\, A_{22}^{-1}\, A_{21}\bigr),
  \qquad
  F(t) := \bigl(B_1 - A_{12}\, A_{22}^{-1}\, B_2\bigr)\, u(t).
\end{equation}
    The unique solution $x\in H^1(0,T;\R^r)$ that satisfies \eqref{eq:Reduced-System} pointwise almost everywhere and $x_1(0) = x_{1,0}\in \R^r$ is given by the variations of constants formula
    \begin{align}\label{eq:varconst}
        x_1(t) = e^{-E_{11}^{-1}S \, t}x_{1,0} + \int_0^t e^{-E_{11}^{-1}S (t-s)} E_{11}^{-1}F(s)\,\mathrm{d}s
    \end{align}
    for all $t\geq 0$. The \emph{reduced Hamiltonian} and output reads: 
\begin{equation}\label{eq:reduced-H}
  \mathcal H(x_1) := \tfrac{1}{2}\, x_1^\top E_{11}\, x_1,
  \qquad 
  \hat{y}_1 = B_1^\top x_1 + B_2^\top \hat{x}_2(x_1,t).
\end{equation}

The matrix $S$ \eqref{eq:S-and-F-def} is generally non-symmetric.  We denote its symmetric
part by
\begin{equation}\label{eq:Stilde-def}
  \widetilde S := \tfrac{1}{2}\bigl(S + S^\top\bigr).
\end{equation}

In the following, we will repeatedly use the following well-established scaling argument.

\begin{remark}[Scaling]\label{rem:scaling}
In general, $\widetilde S$ is only positive semi-definite due to the dissipativity of 
$A$. However, for various results it is convenient to assume that $\widetilde S$ is coercive. This can be achieved by the following scaling argument. Let $x \in H^1(0,T;\R^r)$ solve \eqref{eq:Reduced-System}, and define
\[
 x_\mu(t) = e^{-\mu t} x(t),
\qquad \mu \in \R.
\]
Then $ x_\mu$ solves
\begin{equation}
    E_{11} \frac{\mathrm d}{\mathrm dt} x_\mu(t) + (S+\mu E_{11}) x_\mu(t) = e^{-\mu t}F(t).
\end{equation}

Moreover, since 
$A$ is dissipative, the symmetric part \eqref{eq:Stilde-def}

satisfies $\widetilde S \ge 0$. Therefore,
\begin{equation*}
    \langle x,(S+\mu E_{11})x\rangle
    = \langle x,\widetilde S x\rangle + \mu \langle x,E_{11}x\rangle
    \ge \mu \lambda_{\min}(E_{11})\|x\|^2.
\end{equation*}
Hence, since $E_{11}>0$, the matrix $S+\mu E_{11}$ is coercive for every $\mu>0$. \hfill $\Box$
\end{remark}

\paragraph{Weak formulation.} 
Testing the reduced equation~\eqref{eq:Reduced-System} against
$\p \in L^2(0, T; \R^r)$ 

gives the
following \emph{variational} problem: Find $x_1 \in H^1(0, T; \R^r) \hookrightarrow C^0\bigl([0, T]; \R^r\bigr)$ such that 
\begin{equation}\label{eq:continuous-weak-form}
  \int_0^T \bigl\langle E_{11}\, \ddt x_1(t),\, \p(t) \bigr\rangle
    \diff t
  + \int_0^T \bigl\langle S\, x_1(t),\, \p(t) \bigr\rangle \diff t
  = \int_0^T \bigl\langle F(t),\, \p(t) \bigr\rangle \diff t
  \quad \text{ for all } \p \in L^2(0, T; \R^r).
\end{equation}

\begin{proposition}[Well-posedness and regularity]
\label{prop:existence}
The unique mild solution (given by \eqref{eq:varconst}) is also the unique weak solution of \eqref{eq:continuous-weak-form}. Moreover, the differential variable $x_1$ satisfies the estimate
\begin{equation}\label{eq:stability-continuous}
  \|x_1\|_{C([0,T];\R^r)}^2 + \|x_1\|_{L^2(0,T;\, \R^r)}^{\,2}
  \leq c\left(\|F\|_{L^2(0,T;\R^r)}^{\,2}
     + \|x_{1,0}\|_{E_{11}}^{\,2}\right),
\end{equation}
for some $c\geq 0$
and  
$\|v\|_{E_{11}}^{\,2} := \langle E_{11}\, v, v \rangle$. The algebraic component 
\eqref{eq:algebraic-reconstruction} satisfies
$\hat x_2 \in L^2(0,T;\R^{n-r})$.

If the dynamics~\eqref{eq:Reduced-System} are exponentially stable, that is, if $-E_{11}^{-1}S$ only has eigenvalues with strictly negative real part, then the constant $c\geq 0$ in \eqref{eq:stability-continuous} is uniform in the time horizon $T$.

\end{proposition}

\begin{proof}
By the derivation of the weak formulation, the mild solution is also a weak solution.

Uniqueness of the weak solution follows from 
\eqref{eq:continuous-weak-form} and the fundamental theorem of calculus of variations. Estimate \eqref{eq:stability-continuous} directly follows from 
\eqref{eq:varconst}. The algebraic component satisfies
$\hat x_2(t) = -A_{22}^{-1}(A_{21}\, x_1(t) + B_2\, u(t))$, so
$\hat x_2 \in L^2(0,T;\R^{n-r})$ follows from
$x_1 \in L^2(0,T;\R^r)$ and $u \in L^2(0,T;\R^m)$. In case of exponential stability, i.e., $\|e^{-E_{11}^{-1}St }\|\leq Me^{-\omega t}$ for some $M\geq 1$ and 
$\omega >0$, it follows also from \eqref{eq:varconst} that all expressions on the right-hand side are uniformly bounded, i.e.,
the first term is bounded due to $\|e^{-E_{11}^{-1}St }\| \leq M$,  the second term by integration of an exponentially decaying function on $[0,\, T]$. 

\end{proof}

\color{black}

\subsection{Goal functional and adjoint problem}\label{sec:adjoint}

We define the quantity of interest (QoI) that is of particular interest in the simulation of port-Hamiltonian systems and in our suggested method, which drives the adaptive
time discretizations.

\paragraph{Energy-balance residual as QoI.}
The energy balance~\eqref{eq:energy-balance} holds exactly at the
continuous solution of the pH-DAE~\eqref{eq:ph-dae}.  For any
approximate solution $\tilde x\colon [0,T] \to \R^n$, we measure
the deviation from this balance on a given time grid $\mathcal{T}$
\begin{equation}\label{eq:timegrid}
\mathcal{T}: \quad    0 = t_0 < t_1 < \cdots < t_N = T 
\end{equation}
by the \emph{local
energy-balance residual}
\begin{equation}\label{eq:Iloc-i-full}
  I_{\mathrm{loc}}^i(\tilde x)
  := \biggl|\,
    \int_{t_i}^{t_{i+1}}
      \Bigl[-y(t)^\top u(t)
        + (Q\tilde x(t))^\top R\, Q\tilde x(t)\Bigr]\,\diff t
    + \gHam\bigl(\tilde x(t_{i+1})\bigr)
    - \gHam\bigl(\tilde x(t_i)\bigr)
    \biggr|^2,
\end{equation}
and the \emph{global quantity of interest}
\begin{equation}\label{eq:Iloc-global}
  I_{\mathrm{loc}}(\tilde x)
  := \sum_{i=0}^{N-1} I_{\mathrm{loc}}^i(\tilde x).
\end{equation}
By construction, $I_{\mathrm{loc}}(\tilde x) = 0$ if and only if the
energy balance~\eqref{eq:energy-balance} holds exactly on every
subinterval $[t_i, t_{i+1}]$.  The squaring
in~\eqref{eq:Iloc-i-full}
penalizes larger deviations of the energy
and it renders the
functional differentiable, which is required by the DWR framework.

\paragraph{Reduction to the differential variable.}
Since the algebraic variable is determined by the algebraic
reconstruction~\eqref{eq:algebraic-reconstruction}, the output $y$,
the dissipation, and the Hamiltonian can all be expressed as functions
of~$x_1$ alone.  Substituting $\hat x_2(x_1, t)$
from~\eqref{eq:algebraic-reconstruction} into~\eqref{eq:Iloc-i-full}
recasts each local residual~\eqref{eq:Iloc-i-full} in terms of the
differential variable only.  This motivates the following definitions.

\paragraph{Residual form.}
Let $X := H^1(0, T;\R^r)$ be the trial space for the primal problem.
For $x_1 \in X$ and $\p \in L^2(0,T;\R^r)$, we define the
\emph{residual form} induced by the affine linear bounded operator $\mathcal{A}:X\to L^2(0,T;\R^r)$ defined by
\begin{equation}\label{eq:residual-form}
  \mathcal{A}(x_1)(\p)
  := \int_0^T
    \bigl\langle E_{11}\,\dot x_1(t) + S\, x_1(t) - F(t),
      \;\p(t) \bigr\rangle \,\diff t.
\end{equation}
The primal problem~\eqref{eq:continuous-weak-form} reads equivalently:
find $x_1 \in X$ such that $\mathcal{A}(x_1)(\p) = 0$ for all
$\p \in L^2(0,T;\R^r)$.  Further, the Fr\'echet derivative with respect to $x_1$ in any direction
$v \in X$ is
\begin{equation}\label{eq:bilinear-form}
  \mathcal{A}'(x_1)(v, w)
  = \int_0^T
    \bigl\langle E_{11}\,\dot v(t) + S\, v(t),\; w(t) \bigr\rangle
    \,\diff t
\end{equation}
for all $w \in L^2(0,T;\R^r)$, and is independent of $x_1$.

\paragraph{Reduced goal functional.}
On the time grid $\mathcal{T}$, 
we define the
\emph{reduced goal functional} $\mathcal J : X \to \R$ 
\begin{equation}\label{eq:J-def}
  \mathcal J(x_1)
  := \sum_{i=0}^{N-1} \bigl|\mathcal G_i(x_1)\bigr|^2,
\end{equation}
where the \emph{local energy-balance residual} on $\bI = (t_i,t_{i+1})$
is
\begin{equation}\label{eq:Gi-splitting}
  \mathcal G_i(x_1)
  = \int_{t_i}^{t_{i+1}} g\bigl(t, x_1(t)\bigr)\,\diff t
  + \mathcal H\bigl(x_1(t_{i+1})\bigr)
  - \mathcal H\bigl(x_1(t_i)\bigr)
\end{equation}
 
with the \emph{power imbalance function}
$g$ encoding supply and dissipation terms after reduction, that is,
\begin{equation*}
  g(t, x_1)
  := -\hat y(x_1, t)^\top u(t) + \hat d(x_1, t) \quad
  \text{with} \quad 
  \hat d(x_1, t)
  := \bigl(Q(V x_1 + W \hat x_2)\bigr)^\top R\,
  \bigl(Q(V x_1 + W \hat x_2)\bigr).
\end{equation*}
using the reduced output $\hat{y}$ defined by \eqref{eq:reduced-H} and the \emph{reduced dissipation} $\hat{d}$ (obtained from the kernel splitting pair in $A$). 

Then, $\mathcal G_i$ is quadratic and $\mathcal J$ is quartic in $x_1$.

The reduced goal functional $\mathcal J(x_1)$ \eqref{eq:J-def} is the global QoI~\eqref{eq:Iloc-global} expressed in 
$x_1$, that is, $\mathcal J(x_1) = I_{\mathrm{loc}}(\tilde z)$ when
$\tilde z = V x_1 + W \hat x_2(x_1, \cdot)$.  We work exclusively with
$\mathcal J$ from this point onward, as the reduced formulation is the
natural setting for the variational formulation and error estimation.

\paragraph{Adjoint problem.}
For a direction $v \in X$, the G\^ateaux derivative of $\mathcal J$ \eqref{eq:J-def} is
\begin{equation}\label{eq:dJ-structure}
  \mathcal J'(x_1)(v)
  = 2 \sum_{i=1}^{N} \mathcal G_i(x_1)\, \mathcal G_i'(x_1)(v),
\end{equation}
where the local derivative, obtained by
differentiating~\eqref{eq:Gi-splitting}, reads
\begin{equation}\label{eq:dGi}
  \mathcal G_i'(x_1)(v)
  = \int_{t_{i-1}}^{t_{i}}
      \bigl\langle \nabla_{x_1} g(t, x_1),\, v(t) \bigr\rangle
      \,\diff t
  + \bigl\langle E_{11}\, x_1(t_{i}),\, v(t_{i}) \bigr\rangle
  - \bigl\langle E_{11}\, x_1(t_{i-1}),\, v(t_{i-1}) \bigr\rangle,
\end{equation}
where we have used that $\nabla_{x_1} \mathcal H(x_1) = E_{11}\, x_1$.
Because $\mathcal J'(x_1)(v)$ contains point evaluations, the adjoint
solution subject to this derivative as a source term will not be globally continuous.  Thus, we introduce the space of piecewise absolutely continuous functions
\begin{equation}\label{eq:broken-space}
  Z := \bigl\{\, z \in L^2(0,T;\R^r) \;\big|\;
    z|_{I_i} \in H^1(I_i;\R^r)
    \text{ for } i = 1,\dots,N \bigr\}.
\end{equation}
The \emph{adjoint} $z \in Z$ is defined as the
solution of
\begin{equation}\label{eq:adjoint-weak}
  \mathcal{A}'(x_1)(\p, z)
  = \mathcal J'(x_1)(\p)
  \quad \forall\, \p \in X.
\end{equation}
Note that existence and uniqueness of the adjoint follows from analogous considerations as in \cite[Section 3.1]{BartelSchaller_2025}. At this point, the symbol $z$ refers exclusively to the adjoint variable for the remainder of the paper; the descriptor state of the pH-DAE~\eqref{eq:ph-dae} has been eliminated in favor of the differential variable $x_1$ via the kernel splitting of Section~\ref{sec:formulation}.

\begin{remark}\label{rem:adjoint-at-exact}
At the exact primal solution $x_1$ to \eqref{eq:dJ-structure}, 
the energy balance violation vanishes, that is, $\mathcal G_i(x_1) = 0$, so the right-hand side
of~\eqref{eq:adjoint-weak} vanishes due to \eqref{eq:dJ-structure} and the adjoint is trivially
$z = 0$.  

\end{remark}

Next, we derive the strong form of the
adjoint to make the structure explicit. To this end, we introduce the jump notation for any function $f:[0,\, T]\to X$ and $t_i\in[0,\, T]$
\begin{equation}\label{eq:jumps}
      [f]_i = f(t_i^+) - f(t_i^-) . 
\end{equation}

\begin{lemma}[Strong form of the adjoint]\label{lem:adjoint-strong}
The unique solution $z \in Z$ of the adjoint
problem~\eqref{eq:adjoint-weak} satisfies the following backward-in-time
system:
\begin{subequations}\label{eq:adjoint-strong}
\begin{alignat}{2}
  -E_{11}\, \dot z(t) + S^\top z(t)
    &= \mathbf{f}_z(t),
    &\qquad& t \in \bI,
    \label{eq:dual-strong-ODE}\\
  E_{11}\, [z]_i
    &= \mathbf{j}_z(t_i),
    &\qquad& i = 1,\dots,N{-}1,
    \label{eq:dual-jumps} 
    \qquad 
  E_{11}\, z(T^-)
    = \mathbf{b}_z(T),
\end{alignat}
\end{subequations}

with the source terms
\begin{subequations}\label{eq:adjoint-sources}
\begin{align}
  \mathbf{f}_z(t)
    &:= 2\,\mathcal G_i(x_1)\,
        \nabla_{x_1} g\bigl(t, x_1(t)\bigr),
    \label{eq:fz-def}\\
  \mathbf{j}_z(t_i)
    &:= 2\bigl[\mathcal G_i(x_1)
         - \mathcal G_{i-1}(x_1)\bigr]\,
         E_{11}\, x_1(t_i),
    \label{eq:jz-def} 
    \qquad
  \mathbf{b}_z(T)
    := 2\,\mathcal G_{N-1}(x_1)\,
         E_{11}\, x_1(T).
\end{align}
\end{subequations}

\end{lemma}

\begin{proof}

Let $(\varphi,z)\in H^1(0,T;\R^r)\times Z$ solve \eqref{eq:adjoint-weak}
We apply integration by parts to the time derivative term in~\eqref{eq:bilinear-form} on each subinterval $\bI$ 
\begin{equation*} 
  \int_{t_i}^{t_{i+1}}
    \bigl\langle E_{11}\, \dot\p,\, z \bigr\rangle \,\diff t
  = \bigl\langle \p(t_{i+1}),\, E_{11}\, z(t_{i+1}^-) \bigr\rangle
  - \bigl\langle \p(t_i),\, E_{11}\, z(t_i^+) \bigr\rangle
  - \int_{t_i}^{t_{i+1}}
      \bigl\langle \p,\, E_{11}\, \dot z \bigr\rangle \,\diff t
\end{equation*}

using symmetry of $E_{11}$. Summing 
over $i = 0, \dots, N{-}1$ using that
$\p$ is continuous, we obtain
\begin{align*}
  \sum_{i=0}^{N-1}
  \int_{t_i}^{t_{i+1}}
    \bigl\langle E_{11}\, \dot\p,\, z \bigr\rangle \,\diff t
  &= \bigl\langle \p(T),\, E_{11}\, z(T^-) \bigr\rangle
   - \bigl\langle \p(0),\, E_{11}\, z(0^+) \bigr\rangle
  \notag\\
  &\quad
   - \sum_{i=1}^{N-1}
       \bigl\langle \p(t_i),\, E_{11}\, [z]_i \bigr\rangle
   - \sum_{i=0}^{N-1}
     \int_{t_i}^{t_{i+1}}
       \bigl\langle \p,\, E_{11}\, \dot z \bigr\rangle \,\diff t.
\end{align*}
Adding the $S$ term from~\eqref{eq:bilinear-form} 
and using
$\langle S\, \p, z \rangle = \langle \p, S^\top z \rangle$, the
left-hand side of 
\eqref{eq:adjoint-weak} becomes
\begin{align}\label{eq:adjoint-LHS-expanded}
  \mathcal{A}'(x_1)(\p, z)
  &= \sum_{i=0}^{N-1}
     \int_{t_i}^{t_{i+1}}
       \bigl\langle \p,\,
         -E_{11}\, \dot z + S^\top z \bigr\rangle \,\diff t
  \notag\\
  &\quad
   + \bigl\langle \p(T),\, E_{11}\, z(T^-) \bigr\rangle
   - \bigl\langle \p(0),\, E_{11}\, z(0^+) \bigr\rangle
   - \sum_{i=1}^{N-1}
       \bigl\langle \p(t_i),\, E_{11}\, [z]_i \bigr\rangle.
\end{align}

Next, we expand the right-hand side of the adjoint equation
 $\mathcal J'(x_1)(\p)$ using~\eqref{eq:dJ-structure}
 and~\eqref{eq:dGi}:
\begin{align}\label{eq:adjoint-RHS-expanded}
  \mathcal J'(x_1)(\p)
  &= \sum_{i=0}^{N-1}
     2\, \mathcal G_i
     \int_{t_i}^{t_{i+1}}
       \bigl\langle \nabla_{x_1} g,\, \p \bigr\rangle \,\diff t
+ \sum_{i=1}^{N-1}
       2\bigl(\mathcal G_{i-1} - \mathcal G_i\bigr)
       \bigl\langle E_{11}\, x_1(t_i),\, \p(t_i) \bigr\rangle
  \notag\\
  &\quad
   + 2\, \mathcal G_{N-1}
       \bigl\langle E_{11}\, x_1(T),\, \p(T) \bigr\rangle
   - 2\, \mathcal G_0
       \bigl\langle E_{11}\, x_1(0),\, \p(0) \bigr\rangle.
\end{align}

Equating~\eqref{eq:adjoint-LHS-expanded}
and~\eqref{eq:adjoint-RHS-expanded} for all $\p \in X$ and comparing
the distributed terms, interior-node terms, and boundary terms
separately yields the strong

form~\eqref{eq:dual-strong-ODE}--\eqref{eq:dual-jumps} with
the sources~\eqref{eq:adjoint-sources}.  

\end{proof}

\begin{remark}[Interpretation of the adjoint problem]%
\label{rem:adjoint-interpretation}
The adjoint problem~\eqref{eq:adjoint-strong} is a backward-in-time ODE
with piecewise smooth forcing $\mathbf{f}_z$ on each subinterval.  
At each interior node $t_i$, the adjoint experiences 

a jump proportional to $\mathcal G_i - \mathcal G_{i-1}$, that is, it is only an ODE in the usual sense on each subinterval.

This reflects the mismatch in the energy-balance residual
between adjacent subintervals. 

\end{remark}

\subsection{Norm-augmented weighted quantity of interest}
\label{sec:norm-augmented-QoI-definition}
As alternative to the purely energy-targeted functional $I_{\mathrm{loc}}(\tilde x)$ defined in \eqref{eq:Iloc-global}, we introduce a norm-augmented QoI that balances accuracy in the
discrete energy balance with accuracy of the state trajectory.
Using the underlying Hamiltonian $\mathcal{H}$ and a scalar parameter $\rho\ge 0$, we define
\begin{equation}
  \label{eq:weighted_qoi_norm}
  I_{\mathrm{loc},\rho}(\tilde x)
  :=
  I_{\mathrm{loc}}(\tilde x)
  +
  \rho \int_0^T \mathcal{H}\bigl(\tilde x(t)\bigr)\,\mathrm{d}t.
\end{equation}
The additional term corresponds to an $L^2$-in-time energy functional
and induces the energy norm in the differential variables (that is, $x_1$ after transformation)
\begin{equation*}
  \|\tilde x\|_{\mathcal{E}}^2 := \int_0^T \mathcal{H}\bigl(\tilde x(t)\bigr)\,\mathrm{d}t.
\end{equation*}
The 
dual problem is constructed 
with this modified QoI. 
This choice of the QoI \eqref{eq:weighted_qoi_norm} yields a smooth quadratic contribution to the goal functional and leads to a well-defined adjoint equation.

\section{Discretization}\label{sec:discretization}
We discretize the reduced primal problem~\eqref{eq:Reduced-System} and
the adjoint problem~\eqref{eq:adjoint-strong} using a discontinuous
Galerkin method of order zero in time, denoted by dG(0). 

\subsection{Discretization of the primal problem}
\label{sec:dg0-primal}

For the time grid $\mathcal{T}$ \eqref{eq:timegrid}, we introduce intervals $\kIi = (t_i, t_{i+1}]$, step sizes $k_i := t_{i+1} - t_i$ 
and 
the discrete trial and test space of piecewise constant functions 

\begin{equation}\label{eq:discrete-space}
  X_k := \bigl\{\, v_k \in L^2(0,T;\R^r) \;\big|\;
    v_k|_{\kIi} = v^i \in \R^r,\;
    i = 0,\dots,N{-}1 \bigr\}.
\end{equation}
For $v_k \in X_k$, we write $v_k^i := v_k|_{\kIi}$ for its constant
value on~$\kIi$.  

For
the initial condition, we set $v_k(t_0^-) := x_{1,0}$.

Testing the reduced equation~\eqref{eq:Reduced-System} with $\p_k \in X_k$ on each subinterval $\kIi$ and
applying integration by parts to the time-derivative term produces the
standard dG(0) formulation 
(see,
e.g.,~\cite[Chapter~12]{thomee2006}).  Since $\p_k$ and $x_k$ are both
constant on each~$\kIi$, the scheme reduces to the following 
recursion.

\begin{definition}[Primal dG(0) scheme]\label{def:primal-scheme}
Find $x_k \in X_k$ such that, for $i = 0, \dots, N$,
\begin{equation}\label{eq:primal-step}
  \bigl(E_{11} + k_i\, S\bigr)\, x_k^i
  = E_{11}\, x_k^{i-1}
    + \int_{\bI} F(t)\,\diff t,
\end{equation}
where $x_k^{0} := x_{1,0}$ denotes the initial condition.
\end{definition}

The following auxiliary result provides a convenient estimate for the following. It allows to deduce estimates for the discrete approximation of the original system by estimates from its counterpart for the scaled system subject to coercive governing matrix in the spirit of Remark~\ref{rem:scaling}.
\begin{lemma}\label{lem:scaledtounscaled}
    Given $\mu \ge 0$. Assume that $(x_k),(x_\mu)_k\in X_k$ satisfy
    \begin{align}\label{eq:discret-recursion-unshift-shift}
        \bigl(E_{11} + k_i\, S\bigr)\, x_k^i
  = E_{11}\, x_k^{i-1}
    + \!\!\int_{\bI}\!\! F(t)\,\diff t
,\quad \bigl(E_{11} + k_i\, (S+\mu E_{11})\bigr) (x_\mu)_k^i
  = E_{11} (x_\mu)_k^{i-1}
    + \!\!\int_{\bI}\!\! F(t)\,\diff t
    \end{align}
for $i=1,\ldots,N$, with the same initial value $x_k^0 = (x_\mu)_k^0 = x_{1,0}$. Then there is a constant $c(\mu,\, T) \geq 0$ with $c(\mu,\,T) \stackrel{\mu \to 0} \to 1$ such that
\begin{align}
\begin{split}
   \|x_k^N\|^2 &+ 
   \sum_{i=1}^{N-1} \|x_k^{i+1}-x_k^i\|^2 + \sum_{i=1}^{N} k_i \|x_k^i\|^2 
   \\&\leq c(\mu,\,T) \left(\|(x_\mu)_k^N\|^2 + 
   \sum_{i=1}^{N-1} \|(x_\mu)_k^{i+1}-(x_\mu)_k^i\|^2 + \sum_{i=1}^{N} k_i \|(x_\mu)_k^i\|^2\right). 
   \label{eq:integralestimate}
\end{split}
\end{align}
\end{lemma}
\begin{proof}
    The proof is provided in Appendix~\ref{sec:app}.
\end{proof}

\begin{remark}
\label{rem:primal-wellposed}
The system matrix $E_{11} + k_i\, S$ is generally non-symmetric.

Its symmetric part is
$E_{11} + k_i\, \widetilde S$ is positive definite as $S$ is dissipative, hence has positive semi-definite symmetric part. As $E_{11}$ is symmetric positive definite, $v^\top(E_{11} + k_i S)\,v \geq v^\top E_{11} v > 0$ for all $v \neq 0$, so the
time stepping equation is uniquely solvable at each step and for any $k_i\geq 0$. 
\end{remark}


Note that $(x_\mu)_k$ is \emph{not} the dG(0) discretization of the 
scaled continuous variable from Remark~\ref{rem:scaling}, but rather an 
auxiliary discrete sequence defined by the shifted recursion with the 
original source~$F$.
We briefly recall stability of the dG(0) scheme for our particular setting. We note that this feature is well-known as dG(0)-in-time discretization corresponds to the implicit Euler method, see~\eqref{eq:primal-step} and~\cite{meidner2008}.
\begin{lemma}[Primal stability]\label{lem:stability}
Let Assumption~\ref{ass:ph-structure} hold.
Then, the solution $x_k \in X_k$ of the dG(0)
scheme~\eqref{eq:primal-step} {(with initial data $x_k^{0} := x_{1,0}$)} satisfies the unconditional stability estimate
\begin{equation}\label{eq:discrete-stability}
\begin{split}
    \|x_k^{N}\|^2
  + \sum_{i=1}^{N-1}& \bigl\|x_k^{i+1} - x_k^i\bigr\|^2
  + \sum_{i=1}^{N} k_i \|x_k^i\|^2 
  \\&\leq c(\mu,T) \frac{1}{\min\{(\alpha+\mu),\lambda_{\mathrm{min}}(E_{11})\}} \left(\frac{1}{\alpha+\mu}\int_0^T \|F(t)\|^2\,\diff t
     + \|x_{1,0}\|_{E_{11}}^2\right),
\end{split}
\end{equation}
where
$c(\mu,T)$ is a constant with $c(\mu,T)\stackrel{\mu \to 0}{\to} 1$, $\alpha {\ge 0}$ is the smallest eigenvalue of 
$\widetilde S$ \eqref{eq:Stilde-def}
and $\mu \in {\mathbb{R}^+_0}$ is arbitrary {with $\alpha+\mu>0$}.  In particular, if $\alpha > 0$, we may choose $\mu =0$ such that the upper bound is uniform in the time horizon $T$.
\end{lemma}

\begin{proof}
We introduce 
$S_\mu := S + \mu E_{11}$, which is coercive for any $\mu > 0$ in the sense of Remark~\ref{rem:scaling}, and define $(x_\mu)_k \in X_k$ {as in \eqref{eq:discret-recursion-unshift-shift}
with initial data $(x_k^{0} := x_{1,0}$.})

We test the recursion for $(x_\mu)_k$ in \eqref{eq:discret-recursion-unshift-shift}  with $\varphi^i = (x_\mu)_k^i$ on each subinterval~$\kIi$. 
Rearranging gives
\begin{equation}\label{eq:local-tested}
  \bigl\langle E_{11}\bigl((x_\mu)_k^i - (x_\mu)_k^{i-1}\bigr),\, (x_\mu)_k^i
  \bigr\rangle
  + k_i\, \bigl\langle S_\mu\, (x_\mu)_k^i,\, (x_\mu)_k^i \bigr\rangle
  = \int_{\kIi} \bigl\langle F(t),\, (x_\mu)_k^i \bigr\rangle\,\diff t.
\end{equation}
We sum over $i = 1, \dots, N$ and estimate each group of terms.

\medskip
\noindent
\emph{Step~1 (Time-derivative term).}
The algebraic identity
$\langle A(a{-}b), a \rangle
  = \tfrac12\bigl(\|a\|_A^2 - \|b\|_A^2 + \|a{-}b\|_A^2\bigr)$,
applied with $A = E_{11}$, $a = (x_\mu)_k^i$, $b = (x_\mu)_k^{i-1}$, yields the
telescoping sum
\begin{equation*}
  \sum_{i=1}^{N}
    \bigl\langle E_{11}\bigl((x_\mu)_k^i - (x_\mu)_k^{i-1}\bigr),\, (x_\mu)_k^i \bigr\rangle
  = \tfrac12\bigl(\|(x_\mu)_k^{N}\|_{E_{11}}^2
      - \|x_{1,0}\|_{E_{11}}^2\bigr)
  + \tfrac12 \sum_{i=1}^{N-1}
      \|(x_\mu)_k^{i+1} - (x_\mu)_k^i\|_{E_{11}}^2.
\end{equation*}

\noindent
\emph{Step~2 (Dissipation term).}
Since $\langle S_\mu\, v, v \rangle = \langle (S + \mu E_{11})\, v, v \rangle
\geq (\alpha+\mu)\|v\|^2$ for all $v$, we obtain
\begin{equation*}
  \sum_{i=1}^{N} k_i\,\langle S_\mu\, (x_\mu)_k^i,\, (x_\mu)_k^i \rangle
  \geq (\alpha+\mu) \sum_{i=1}^{N} k_i\, \|(x_\mu)_k^i\|^2.
\end{equation*}

\noindent
\emph{Step~3 (Right-hand side).}
By Young's inequality with parameter $\varepsilon = \alpha+\mu$,
\begin{equation*}
  \sum_{i=1}^{N} \int_{I_i}
    \bigl\langle F(t),\, (x_\mu)_k^i \bigr\rangle\,\diff t
  \leq \frac{1}{2(\alpha+\mu)}\int_0^T \|F(t)\|^2\,\diff t
     + \frac{\alpha+\mu}{2} \sum_{i=1}^{N} k_i\,\|(x_\mu)_k^i\|^2.
\end{equation*}

\noindent
\emph{Step~4 (Assembly and resubstiution).}
Combining the above steps and rearranging gives

\begin{align*}
  \tfrac12\|(x_\mu)_k^{N}\|_{E_{11}}^2
  + \tfrac12 \!\sum_{i=1}^{N-1}\|(x_\mu)_k^{i+1} \!- (x_\mu)_k^i\|_{E_{11}}^2
  &+ \tfrac{\alpha+\mu}{2}\sum_{i=1}^{N} k_i\,\|(x_\mu)_k^i\|^2
  \\&\leq \tfrac12\|x_{1,0}\|_{E_{11}}^2
     + \frac{1}{2(\alpha+\mu)}\!\int_0^T \!\!\|F(t)\|^2\,\diff t.
\end{align*}
Dividing by $\tfrac12\min\{(\alpha+\mu),\,\lambda_{\min}(E_{11})\}$ to convert all $E_{11}$-norms on the left to Euclidean norms yields~\eqref{eq:discrete-stability} for the shifted variable $(x_{\mu})_k$ (without $c(\mu,T)$).

Using Lemma~\ref{lem:scaledtounscaled} (assumptions are fulfilled), we obtain for the unscaled variable estimate~\eqref{eq:discrete-stability}.

\end{proof}
\color{black}

The following result is a standard result from the literature.
\begin{lemma}[A priori error estimate]\label{lem:apriori}
Let $x_1 \in H^1(0,T;\R^r)$ be the exact solution of the reduced
equation~\eqref{eq:Reduced-System} and let $x_k$ be its dG(0)
approximation with exact initial data $x_k^0 = x_1(0)$. Then
\begin{equation}\label{eq:apriori}
  \|x_k - x_1\|_{L^2(0,T;\R^n)}
  \le C\, k\, \|\dot x_1\|_{L^2(0,T;\R^n)},
\end{equation}
where $k := \max_i k_i$ and $C>0$.
In particular, the dG(0) scheme is first-order accurate in time. If the system is exponentially stable, $C>0$ is also independent of the time horizon $T$.
\end{lemma}

\begin{proof}
This is a standard a priori error estimate for the dG(0) time
discretization of linear evolution equations with coercive main operator which may be always achieved by scaling in view of Remark~\ref{rem:scaling}.
In finite dimensions, coercivity of the operator is equivalent to
positive definiteness of its symmetric part, and the result follows
directly from the general theory, see~\cite[Theorem~5.5]{meidner2008}. 

\end{proof}
\color{black}

\subsection{Corresponding discrete adjoint problem}\label{sec:adjoint-discrete}

We discretize the adjoint problem on the same partition $\mathcal{T}$ and in the same
space $X_k$ \eqref{eq:discrete-space} as the primal problem. The discrete adjoint is obtained by
transposing the block structure of the primal system, which
automatically produces a backward-in-time recurrence.
In fact, the primal scheme~\eqref{eq:primal-step} defines, over all $N$
intervals, a block lower-bidiagonal linear system
$\mathbf{A}\,\mathbf{x} = \mathbf{b}$, with
\begin{equation}\label{eq:primal-block}
  \mathbf{A} =
  \begin{pmatrix}
    E_{11} + k_1 S &                 &        &                      \\
    -E_{11}        & E_{11} + k_2 S  &        &                      \\
                    \mbox{}\hspace*{10ex}\ddots           & \mbox{}\hspace*{10ex}\ddots &                      \\
                                 & -E_{11} & E_{11} + k_{N} S
  \end{pmatrix}.
\end{equation}
Thus, the discrete adjoint system is the transposed problem
$\mathbf{A}^{\!\top} \mathbf{z} = \mathbf{R}$, which 
reads ($E_{11}^\top=E_{11}$)
\begin{equation}\label{eq:adjoint-block}
  \begin{pmatrix}
    E_{11} + k_1 S^\top & -E_{11}              &        &  \\
                        & E_{11} + k_2 S^\top   & \ddots &  \\
                        &                       & \ddots & -E_{11} \\
                        &                       &        & E_{11} + k_{N}S^\top
  \end{pmatrix}
  \begin{pmatrix} z_k^1 \\ z_k^2 \\ \vdots \\ z_k^{N} \end{pmatrix}
  =
  \begin{pmatrix}
    \mathbf{R}_1 \\ \mathbf{R}_2 \\ \vdots \\ \mathbf{R}_{N}
  \end{pmatrix},
\end{equation}
where the right-hand side $\mathbf{R}_i$ is the derivative of the goal
functional $\mathcal J$ \eqref{eq:J-def} with respect to~$x_k^i$:
\begin{equation}\label{eq:discrete-source}
  \mathbf{R}_i
  = 2\,\mathcal G_i(x_k)
    \int_{I_i} \nabla_{x_1} g\bigl(t, x_k^i\bigr)\,\diff t
  + 2\bigl(\mathcal G_i(x_k) - \mathcal G_{i+1}(x_k)\bigr)\,
    E_{11}\, x_k^i,
  \qquad i = 1,\dots,N,
\end{equation}
with the convention $\mathcal G_{N+1} := 0$.

The backward recurrence emerging from the system~\eqref{eq:adjoint-block}:
\begin{definition}[Discrete adjoint scheme]\label{def:adjoint-scheme}
Set $z_k^N := 0$.  For $i = N, \dots, 1$ and $\mathbf{R}_i$ given by~\eqref{eq:discrete-source}, solve
\begin{equation}\label{eq:dual-step}
  \bigl(E_{11} + k_i\, S^\top\bigr)\, z_k^i
  = E_{11}\, z_k^{i+1} + \mathbf{R}_i.
\end{equation}

\end{definition}

\begin{remark}
\label{rem:adjoint-wellposed}
The system matrix $E_{11} + k_i\, S^\top$ has symmetric part
$E_{11} + k_i\, \widetilde S$, which is positive definite as $\widetilde S\geq 0$ and $E_{11} >0$. Hence, each step of
the backward recurrence~\eqref{eq:dual-step} is uniquely solvable.

\end{remark}

\begin{lemma}[Adjoint stability]\label{lem:adjoint-stability}
Let Assumption~\ref{ass:ph-structure} hold. {Then $z_k$~\eqref{eq:dual-step} satisfies}

\begin{equation}\label{eq:adjoint-discrete-stability}
  \|z_k^1\|^2
  + \sum_{i=1}^{N-1} \bigl\|z_k^{i+1} - z_k^i\bigr\|^2
  + \sum_{i=1}^{N-1} k_i\, \|z_k^i\|^2
  \leq \frac{1}{\min\{(\alpha+\mu)/2,\lambda_{\mathrm{min}}(E_{11})\}} \frac{c(\mu,T)}{\alpha+\mu}
       \sum_{i=1}^{N} \frac{\|\mathbf{R}_i\|^2}{k_i}
\end{equation}
where
$c(\mu,T)$ is a constant with $c(\mu,T)\stackrel{\mu \to 0}{\to} 1$, $\alpha  \ge 0 \color{black}$ is the smallest eigenvalue of
$\widetilde{S}$ \eqref{eq:Stilde-def}
and $\mu \in  \mathbb{R}_0^{+}$ is arbitrary with $\alpha + \mu > 0$\color{black}.  In particular, if $\alpha > 0$, we may choose $\mu =0$ 
{and $c$ is independent of $T$}. 
\end{lemma}
\begin{proof}
The argument mirrors the 
proof of Lemma~\ref{lem:stability}. We use 
$S_\mu := S + \mu E_{11}$, which is coercive for any $\mu > 0$, and
define $(z_\mu)_k \in X_k$ as the solution of the {corresponding backward recursion:} 

\begin{equation*}
  E_{11}\bigl((z_\mu)_k^i - (z_\mu)_k^{i+1}\bigr)
  + k_i\, S_\mu^\top\, (z_\mu)_k^i
  = \mathbf{R}_i.
\qquad i = N{-}1, \dots, 1,
\end{equation*}
with $(z_\mu)_k^N := 0$ and $\mathbf{R}_i$ \eqref{eq:discrete-source}, i.e., the same right-hand side as in the unscaled case \eqref{eq:dual-step}. \color{black}

Testing with $(z_\mu)_k^i$ and using the identity from Step~1 of the proof of Lemma~\ref{lem:stability} with $a = (z_\mu)_k^i$, $b = (z_\mu)_k^{i+1}$, together with the coercivity bound from Step~2 of Lemma~\ref{lem:stability}, yields
\begin{equation*}
  \tfrac12\bigl(\|(z_\mu)_k^i\|_{E_{11}}^2
    - \|(z_\mu)_k^{i+1}\|_{E_{11}}^2
    + \|(z_\mu)_k^i - (z_\mu)_k^{i+1}\|_{E_{11}}^2\bigr)
  + k_i\,(\alpha+\mu)\,\|(z_\mu)_k^i\|^2
  \leq \bigl\langle \mathbf{R}_i,\, (z_\mu)_k^i \bigr\rangle.
\end{equation*}
Summing over $i = 1, \dots, N{-}1$ and using $(z_\mu)_k^N = 0$ on the
left-hand side, the telescope sum gives
\begin{equation}\label{eq:adjoint-stab-intermediate}
  \tfrac12\,\|(z_\mu)_k^1\|_{E_{11}}^2
  + \tfrac12 \sum_{i=1}^{N-1}
      \|(z_\mu)_k^{i+1} - (z_\mu)_k^i\|_{E_{11}}^2
  + (\alpha+\mu) \sum_{i=1}^{N-1} k_i\, \|(z_\mu)_k^i\|^2
  \leq \sum_{i=1}^{N-1}
    \bigl\langle \mathbf{R}_i,\, (z_\mu)_k^i \bigr\rangle.
\end{equation}
For the right-hand side, we apply
Young's inequality with parameter $\varepsilon = \alpha+\mu$:
\begin{equation*}
  \bigl\langle \mathbf{R}_i,\, (z_\mu)_k^i \bigr\rangle
  = \Bigl\langle \frac{\mathbf{R}_i}{\sqrt{k_i}},\;
      \sqrt{k_i}\, (z_\mu)_k^i \Bigr\rangle
  \leq \frac{1}{2(\alpha+\mu)}\,
       \frac{\|\mathbf{R}_i\|^2}{k_i}
     + \frac{\alpha+\mu}{2}\, k_i\, \|(z_\mu)_k^i\|^2.
\end{equation*}
{We substitute  into~\eqref{eq:adjoint-stab-intermediate}, rearrange and
divide by $\tfrac12\min\{(\alpha+\mu),\,\lambda_{\min}(E_{11})\}$ to convert
$E_{11}$-norms to Euclidean norms. This }
gives~\eqref{eq:adjoint-discrete-stability} for the shifted adjoint $(z_\mu)_k$.

Since $S^\top$ has the same symmetric part $\widetilde{S}$ as $S$, Lemma~\ref{lem:scaledtounscaled} applies to the transposed backward system, and the estimate~\eqref{eq:adjoint-discrete-stability} for $z_k$ follows.
\end{proof}
\color{black}

\begin{remark}[Scaling of the source norm]
\label{rem:source-scaling}
The bound~\eqref{eq:adjoint-discrete-stability} involves the weighted
sum $\sum_i \|\mathbf{R}_i\|^2/k_i$.  The discrete
source~\eqref{eq:discrete-source} contains an integral term of size
$O(k_i)$ and an endpoint  
term of size $O(1)$, so
$\|\mathbf{R}_i\|^2/k_i = O(1/k_i)$ on intervals where the
energy-balance residual changes abruptly.  This is consistent with the
structure of the continuous adjoint, which involves nodal jump
conditions. 
In a piecewise constant time discretization, such nodal contributions
naturally lead to the observed $O(1/k_i)$ scaling.
\end{remark}

\section{A~posteriori error estimation and adaptive algorithm}
\label{sec:DWR-analysis}
The analysis follows the DWR framework 
as developed in~\cite{Rannacher2003}, using
the Lagrangian formalism to handle the quartic structure of
$\mathcal J$.  We specialize the abstract identity to the dG(0)
discretization, derive computable local error indicators, and present
the complete adaptive loop.

\subsection{Error representation and residuals for dG(0)}\label{sec:abstract-DWR}

We define the \emph{Lagrangian} $\mathcal L\colon X \times Z \to \R$ by
\begin{equation}\label{eq:lagrangian}
  \mathcal L(x_1, z) := \mathcal J(x_1) - \mathcal{A}(x_1)(z),
\end{equation}
with primal trial space $X$,
adjoint space $Z$, \eqref{eq:broken-space}, 
goal
functional $\mathcal J$, \eqref{eq:J-def}, and residual
form $\mathcal{A}$, \eqref{eq:residual-form}.

At the exact primal solution $x_1$ we have $\mathcal{A}(x_1)(\cdot) = 0$, such that
$\mathcal L(x_1, z) = \mathcal J(x_1)$ and the pair $(x_1, z)$ is a
stationary point of $\mathcal L$.
At the discrete level, the dG(0) solution $x_k \in X_k$ satisfies
$\mathcal{A}(x_k)(\p_k) = 0$ for all $\p_k \in X_k$, and the discrete adjoint
$z_k \in X_k$ satisfies
$\mathcal{A}'(x_k)(\p_k, z_k) = \mathcal J'(x_k)(\p_k)$ for all
$\p_k \in X_k$.  These are the discrete stationarity conditions of
$\mathcal L$.

\begin{theorem}[DWR error identity]\label{thm:error-rep}
Let $(x_1, z) \in X \times Z$ solve the continuous primal \eqref{eq:continuous-weak-form} and adjoint
problems \eqref{eq:adjoint-weak}, and let $(x_k, z_k) \in X_k \times X_k$ be the dG(0)
approximations with according residuals
\begin{align}\label{eq:residuals-def}
  \rho(x_k)(w) &:= -\mathcal{A}(x_k)(w),
  \qquad 
  \rho^*(x_k, z_k)(v)
    := \mathcal J'(x_k)(v) - \mathcal{A}'(x_k)(v, z_k).
\end{align}
Then, for any $\p_k, \psi_k \in X_k$,
\begin{equation}\label{eq:DWR-identity}
  \mathcal J(x_1) - \mathcal J(x_k)
  = \tfrac12\,\rho(x_k)(z - \p_k)
  + \tfrac12\,\rho^*(x_k, z_k)(x_1 - \psi_k)
  + \mathcal R^{(3)},
\end{equation}
where the remainder is cubic in the primal error $e := x_1 - x_k$:
\begin{equation*}
  \mathcal R^{(3)}
  = \tfrac12 \int_0^1
      \mathcal J'''\bigl(x_k + s\,e\bigr)(e, e, e)\;
      s(s{-}1)\,\diff s
  = O(\|e\|_X^3).
\end{equation*}
\end{theorem}

\begin{proof}
For a proof, see \cite[Proposition 6.2]{Rannacher2003}. 
\end{proof}

We evaluate the abstract residuals~\eqref{eq:residuals-def} using the
piecewise constant structure of the dG(0) discretization.  Since $x_k$
and $z_k$ are constant on each interval, their time derivatives vanish
in the interior, and the residuals decompose into interior contributions
(from the algebraic mismatch with the right-hand side) and nodal
contributions (from the jumps at the time nodes).

\begin{lemma}[Primal residual]\label{lem:explicit-primal}
For any $w \in X$, the primal residual~(\ref{eq:residuals-def})
decomposes as
\begin{equation}\label{eq:primal-residual-explicit}
  \rho(x_k)(w)
  = \sum_{i=1}^{N} \bigl[
      \mathcal R_{\mathrm{int}}^i(w)
    + \mathcal R_{\mathrm{jump}}^i(w)
    \bigr],
\end{equation}
with, using the jump notation $[x_k]_i$ as defined in \eqref{eq:jumps},
\begin{align}\label{eq:primal-residual-components}
  \mathcal R_{\mathrm{int}}^i(w)
    &:= \int_{t_{i-1}}^{t_{i}}
           \bigl\langle F(t) - S\, x_k^i,\; w(t) \bigr\rangle
         \,\diff t,
    \qquad 
  \mathcal R_{\mathrm{jump}}^i(w)
    := -\bigl\langle E_{11}\,[x_k]_i,\; w(t_{i-1}^+) \bigr\rangle,
\end{align}
where $[x_k]_i := x_k^i - x_k^{i-1}$ with $x_k^0 = x_{1,0}$.
\end{lemma}

\begin{proof}
By definition,
$\rho(x_k)(w)
= -\mathcal{A}(x_k)(w)
= -\int_0^T \langle E_{11}\dot x_k + S\,x_k - F,\; w\rangle\,\diff t$.
On each open interval~$I_i = (t_{i-1}, t_i]$, $x_k$ is constant, so $\dot x_k = 0$
and the interior contribution reduces to
$\int_{I_i}\langle F - S\,x_k^i,\; w\rangle\,\diff t$.
The distributional time derivative of the piecewise constant function
$x_k$ is
$E_{11}\dot x_k = \sum_{i=1}^{N} E_{11}[x_k]_i\,\delta_{t_{i-1}}$,
where $\delta_{t_{i-1}}$ denotes the Dirac delta, that is, point evaluation at~$t_{i-1}$.  Testing
against the continuous function $w$ produces the nodal
terms $\langle E_{11}[x_k]_i,\, w(t_{i-1})\rangle$.  Collecting all
contributions with the appropriate signs
yields~\eqref{eq:primal-residual-components}.
\end{proof}

\begin{lemma}[Dual residual]\label{lem:explicit-dual}
For any $v \in X$, the dual residual in~(\ref{eq:residuals-def})
decomposes as
\begin{gather}\label{eq:dual-residual-explicit}
  \rho^*(x_k, z_k)(v)
  = \sum_{i=1}^{N} \mathcal R_{\mathrm{int}}^{*,i}(v)
  + \sum_{i=2}^{N} \mathcal R_{\mathrm{jump}}^{*,i}(v)
  + \mathcal R_{\mathrm{bdy}}^{*}(v),
\\
  \text{with} \qquad \mathcal R_{\mathrm{int}}^{*,i}(v)
    := \int_{t_{i-1}}^{t_{i}}
           \bigl\langle \mathbf{f}_z(t) - S^\top z_k^i,\;
             v(t) \bigr\rangle\,\diff t,
             \qquad
  \mathcal R_{\mathrm{jump}}^{*,i}(v)
    := \bigl\langle
           E_{11}(z_k^{i-1} - z_k^i)
           - \mathbf{j}_z(t_{i-1}),\;
           v(t_{i-1}) \bigr\rangle,
  \quad\, \mbox{}
  \nonumber
  \\
  \mathcal R_{\mathrm{bdy}}^{*}(v)
    := \bigl\langle
           \mathbf{b}_z(T) - E_{11}\, z_k^{N},\;
           v(T) \bigr\rangle,
   \label{eq:dual-residual-components}
\end{gather}
where $\mathbf{f}_z$, $\mathbf{j}_z$, and $\mathbf{b}_z$ are the
continuous adjoint right-hand side terms defined
in~\eqref{eq:adjoint-sources}.
\end{lemma}

\begin{proof}
By definition,
$\rho^*(x_k, z_k)(v) = \mathcal J'(x_k)(v) - \mathcal{A}'(x_k)(v, z_k)$.
We inspect the derivative term inside $\mathcal{A}'(x_k)(v, z_k)$
(cf.~\eqref{eq:bilinear-form}). 
Since $z_k$ is piecewise 
constant, $\dot z_k = 0$, on each $\kIi$, the integration by parts
produces the interior term
$\langle S\,v, z_k^i\rangle = \langle v, S^\top z_k^i\rangle$
together with boundary terms at the nodes.
Collecting terms together and using the continuity of $v$ at the nodes gives
\begin{align*}
  \mathcal{A}'(x_k)(v, z_k)
  &= \sum_{i=1}^{N}
       \int_{I_i} \bigl\langle v,\, S^\top\! z_k^i \bigr\rangle\,\diff t
   + \bigl\langle v(T),\, E_{11}\, z_k^{N} \bigr\rangle
   - \bigl\langle v(0),\, E_{11}\, z_k^1 \bigr\rangle
   + \sum_{i=2}^{N}
       \bigl\langle v(t_{i-1}),\,
         E_{11}(z_k^i \!-\! z_k^{i-1}) \bigr\rangle.
\end{align*}
Subtracting this equality 
from $\mathcal J'(x_k)(v)$
(expanded in~\eqref{eq:adjoint-RHS-expanded}) and grouping the terms
yields~\eqref{eq:dual-residual-components}.  
\end{proof}

\begin{remark}
\label{rem:dual-residual-role}
The dual residual $\rho^*$ measures how well the discrete adjoint
$z_k$ satisfies the continuous adjoint equation.  In the standard DWR
simplification (see below Section~\ref{sec:computable-estimator}), this term is
neglected.  This is justified because the dual residual involves the
product of the adjoint discretization error and the primal
interpolation error, making it higher-order relative to the primal
residual.
\end{remark}

\subsection{Computable error estimator and goal-oriented adaptive algorithm}
\label{sec:computable-estimator}

The error identity~\eqref{eq:DWR-identity} involves the unknown exact
solutions $x_1$ and~$z$.  We apply three standard
ramifications~\cite{Rannacher2003, meidner2008} to obtain a computable quantity.

\paragraph{Step~1: Neglect cubic remainder.}
The term $\mathcal R^{(3)} = O(\|e\|_X^3)$ is asymptotically negligible
for sufficiently fine time meshes.

\paragraph{Step~2: Drop dual residual.}
The dual residual is
of higher order than the primal residual.  Neglecting it yields
\begin{equation}\label{eq:simplified-estimate}
  \mathcal J(x_1) - \mathcal J(x_k)
  \approx \rho(x_k)(z - z_k).
\end{equation}

\paragraph{Step~3: Approximate adjoint error weight.}
Following~\cite{EndtmayerLangerWick2020}, the exact adjoint error $z - z_k$
is replaced by 
\begin{equation}\label{eq:weight-function}
  w_k(t) := \tilde z_k(t) - z_k(t),
\end{equation}
where $\tilde z_k(t)$ is a 
piecewise linear ({and globally} continuous) interpolant with nodal values 
\begin{equation}\label{eq:adjoint-reconstruction}
  \tilde z_k(t_i) :=
  \begin{cases}
    z_k^1
      & \text{if } i = 0,\\[3pt]
    \tfrac12\bigl(z_k^{i} + z_k^{i+1}\bigr)
      & \text{if } 1 \leq i \leq N{-}1,\\[3pt]
    0
      & \text{if } i = N.
  \end{cases}
\end{equation}

Substituting the weight~\eqref{eq:weight-function} into the primal
residual~\eqref{eq:primal-residual-explicit} yields the computable
indicator. 
{For the nodal values of $w_k$, we obtain}

\begin{equation}\label{eq:weight-nodal}
  w_k(t_{i-1}^+)
  = \tilde z_k(t_{i-1}) - z_k^i
  = -\tfrac12\,\Delta z_i,
  \qquad
  w_k(t_{i}^-)
  = \tilde z_k(t_{i}) - z_k^i
  = \tfrac12\,\Delta z_{i+1},
\end{equation}
where $\Delta z_j := z_k^j - z_k^{j-1}$ denotes the adjoint jump at
node $t_{j-1}$ (with the convention $\Delta z_1 = 0$ from the boundary
rule at $i = 0$, so that $w_k(t_0^+) = 0$).

\begin{definition}[Local error indicator]\label{def:indicator}
The \emph{signed} error indicator on interval $I_i$ is
\begin{equation}\label{eq:computable-indicator}
  \eta_i
  :=
       \mathcal R_{\mathrm{int}}^i(w_k)
     + \mathcal R_{\mathrm{jump}}^i(w_k).
\end{equation}
The global estimator and the refinement indicator are
\begin{equation}\label{eq:global-and-marking}
  \eta_{\mathrm{tot}}
  := \Bigl|\sum_{i=1}^{N} \eta_i\Bigr|,
  \qquad
  \bar\eta_i := |\eta_i|.
\end{equation}
\end{definition}

\begin{remark}[Signed vs.\ absolute indicators]
  The global stopping criterion 
  $\eta_{\mathrm{tot}} = |\sum_i \eta_i|$  
  respects the error
  identity~\eqref{eq:DWR-identity} and permits cancellation {of} 
  positive and negative $\eta_i$. 
  Using
  $\sum_i |\eta_i|$ instead would inflate the estimator.  The absolute
  values~$\bar\eta_i = |\eta_i|$ 
  are solely {needed} for D\"orfler marking, where the goal is to identify the
  intervals with the largest error contribution regardless of sign.
\end{remark}

\paragraph{Algebraic evaluation and adaptive refinement algorithm.}
Since $w_k$ is linear on $I_i$, each component
of~\eqref{eq:computable-indicator} can be evaluated explicitly.
For the interior term, $\mathcal{R}^i_{\mathrm{int}}$ \eqref{eq:primal-residual-components},
the constant part
$-S\,x_k^i$ integrates exactly against the linear weight:
\begin{equation}\label{eq:constant-part}
  \int_{I_i}\bigl\langle -S\,x_k^i,\; w_k(t)\bigr\rangle\,\diff t
  = \bigl\langle -S\,x_k^i,\;
      \tfrac{k_i}{2}\bigl(w_k(t_{i-1}^+) + w_k(t_{i}^-)\bigr)
    \bigr\rangle
  = \bigl\langle -S\,x_k^i,\;
      \tfrac{k_i}{4}(\Delta z_{i+1} - \Delta z_i)
    \bigr\rangle.
\end{equation}
The right-hand side $F(t) = (B_1 - A_{12}A_{22}^{-1}B_2)\,u(t)$
contribution is evaluated using
the interval average
$\bar F_i := k_i^{-1}\int_{I_i}F(t)\,\diff t$ (or an equivalent
low-order quadrature).
With this approximation,
\begin{equation}\label{eq:interior-evaluated}
  \mathcal R_{\mathrm{int}}^i(w_k)
  \approx \bigl\langle \bar F_i - S\,x_k^i,\;
    \tfrac{k_i}{4}(\Delta z_{i+1} - \Delta z_i)
  \bigr\rangle.
\end{equation}
The jump term, $\mathcal{R}^i_{\mathrm{jump}}$ \eqref{eq:primal-residual-components}, is evaluated at $t_{i-1}^+$ using the
nodal value from~\eqref{eq:weight-nodal}:
\begin{equation}\label{eq:jump-evaluated}
  \mathcal R_{\mathrm{jump}}^i(w_k)
  = -\bigl\langle E_{11}[x_k]_i,\; w_k(t_{i-1}^+) \bigr\rangle
  = \tfrac12\bigl\langle E_{11}(x_k^i - x_k^{i-1}),\;
      \Delta z_i \bigr\rangle.
\end{equation}
Combining~\eqref{eq:interior-evaluated}
and~\eqref{eq:jump-evaluated}, the indicator takes the algebraic form
\begin{equation}\label{eq:algebraic-indicator}
  \boxed{\eta_i
  =
      \bigl\langle
        \bar F_i - S\, x_k^i,\;
        \tfrac{k_i}{4}(\Delta z_{i+1} - \Delta z_i)
      \bigr\rangle
    + \tfrac12\bigl\langle
        E_{11}(x_k^i - x_k^{i-1}),\;
        \Delta z_i
      \bigr\rangle
    + O\bigl(k_i^2\bigr).}
\end{equation}
The first term measures the mismatch between the right-hand side and
the discrete operator, weighted by the adjoint sensitivity.  The second
term penalizes large state jumps in regions of high adjoint sensitivity.
Together, refinement is concentrated on intervals where the
discretization error and the sensitivity of the goal functional are
simultaneously large.

\begin{remark}[Simplified indicator]\label{rem:simplified-indicator}
If only the leading adjoint difference is retained (i.e.,
$\Delta z_{i+1} - \Delta z_i$ is replaced by $\Delta z_i$ and adjusting the prefactor), the indicator reduces to
\begin{equation}\label{eq:simplified-indicator}
  \eta_i
  \approx
    \tfrac{k_i}{2}\bigl\langle \bar F_i - S\,x_k^i,\;
      \Delta z_i \bigr\rangle
    + \bigl\langle E_{11}[x_k]_i,\; \Delta z_i \bigr\rangle.
\end{equation}
This is computationally cheaper and sufficient for driving refinement
in practice, as it retains the correct scaling and captures the dominant
contribution from the state jump.
\end{remark}

\begin{remark}
\label{rem:indicator-boundary}
By the boundary rule $\tilde z_k(t_0) = z_k^1$ in~\eqref{eq:adjoint-reconstruction}, the weight satisfies $w_k(t_0^+) = 0$, so the jump term~\eqref{eq:jump-evaluated} vanishes at the first interval ($i=1$). This is consistent with the fact that the initial condition is prescribed exactly ($x_k^{0} = x_{1,0}$): the state jump $x_k^1 - x_k^0$ reflects the mismatch between the initial data and the first computed step, but carries zero adjoint weight at~$t_0$.
At the last interval ($i = N$), the interior term~\eqref{eq:interior-evaluated} involves $\Delta z_{N+1}$. Since the adjoint is computed for $i = N{-}1, \dots, 1$ with terminal value $z_k^N = 0$, we extend the convention by setting $z_k^{N+1} := 0$, so that $\Delta z_{N+1} = 0 - z_k^N = 0$. Consequently, the interior contribution at $i = N$ reduces to $-\frac{k_N}{4}\langle \bar F_N - S\,x_k^N,\, \Delta z_N \rangle$, and the terminal source from the continuous adjoint, already absorbed into~$\mathbf{R}_N$, is not double-counted.
\end{remark}

Finally, the complete adaptive procedure is given in
Algorithm~\ref{alg:adaptivity}.

\begin{algorithm}[hbt]
\caption{Goal-oriented adaptivity for linear index-$1$ pH-DAEs}
\label{alg:adaptivity}
\begin{algorithmic}[1]
\Require initial grid $\mathcal T_0$, tolerance $\mathit{TOL}$,
         D\"orfler parameter $\theta \in (0,1)$.
\Loop
  \State
    \textbf{compute} $x_k$ via forward
    scheme~\eqref{eq:primal-step}. \hfill  \# primal solve
  \State
    \textbf{compute} $z_k$ via backward
    scheme~\eqref{eq:dual-step}. \hfill \# adjoint solve
  \State
    \textbf{evaluate} $w_k = \tilde z_k - z_k$ via 
    \eqref{eq:adjoint-reconstruction}. \hfill \# reconstruct
  \State
    \textbf{evaluate} signed indicators $\eta_i$
    \eqref{eq:algebraic-indicator}. \hfill  \# estimate
    \State
    \textbf{evaluate} global estimator
    $\eta_{\mathrm{tot}} = \bigl|\sum_{i=1}^{N} \eta_i\bigr|$.
  \If{$\eta_{\mathrm{tot}} \leq \mathit{TOL}$}
    \State \textbf{accept} the current solution and \textbf{stop}.
  \EndIf
  \State
    \textbf{select} a minimal set $\mathcal M$ satisfying
    $\sum_{I_i \in \mathcal M} |\eta_i|
     \geq \theta\sum_{i=1}^{N}|\eta_i|$.
    \hfill \# D\"orfler marking
  \State
    \textbf{bisect} all intervals in $\mathcal M$. \hfill \# refine
\EndLoop
\end{algorithmic}
\end{algorithm}

\begin{remark}
\label{rem:signed-vs-abs}
Notice, the global stopping criterion based on $\bigl|\sum_{i=1}^{N} \eta_i\bigr|$ respects the error identity~\eqref{eq:simplified-estimate}. 
The D\"orfler marking $|\eta_i|$ locates the intervals with the
largest magnitude of error contribution, regardless of signs. 
This
two-level strategy is standard practice in the DWR
literature~\cite{Rannacher2003, EndtmayerLangerWick2020}.
\end{remark}

Before we validate the goal oriented refinement numerically, we present an algorithmic enhancement via parallel computations and assess its properties.

\section{Parallel computation of adjoint approximations}

\label{sec:efficient_adjoint}

 The solution of the adjoint equation constitutes the main computational overhead of the proposed adaptive method. To alleviate the computational cost necessary for evaluation of the error indicators, we suggest a computationally cheap approximation of this adjoint system.  The backward recurrence~\eqref{eq:dual-step} resembles a backwards-in-time discrete-time system, hence intervals are coupled to the neighboring ones.  
In this part, we show that a \emph{strictly} dissipative structure of the pH-DAE allows for dropping the inter-interval coupling, enabling full parallelized computation of the adjoint. We show that due to the port-Hamiltonian structure, the Block-Jacobi iteration is a strict contraction. 

Intuitively, this parallelization property and block diagonalization leverages that error transport is not global, and it suffices to only consider neighboring subsystems. A similar stability property was used in \cite{grune2022efficient} to enable local refinements in predictive control algorithms.

We recall the linear system for the discrete adjoint problem  $\mathbf{A}^{\!\top} \mathbf{z} = \mathbf{R}$, which is explicitly stated in~\eqref{eq:adjoint-block}. 

Now, we decompose $\mathbf{A}^{\!\top} = \mathcal D - \mathcal U$, where
$\mathcal D := \mathrm{blockdiag}(M_0, \dots, M_{N-1})$ with $M_i:=E_{11}+k_i S^\top $ and
$\mathcal U$ contains the super-diagonal blocks $E_{11}$.  The
\emph{Block-Jacobi iteration} for  $\mathbf{Z}^{(l)}$, $l=0,1,\dotsc$, reads
\begin{equation}\label{eq:jacobi-iteration}
  \mathcal D\,\mathbf{Z}^{(\ell+1)}
  = \mathcal U\,\mathbf{Z}^{(\ell)} + \mathbf{R},
  \qquad \mathbf{Z}^{(0)} = \mathbf{0}.
\end{equation}
This decouples across all intervals: each local system
$M_i\, z_k^{i,(\ell+1)} = E_{11}\, z_k^{i+1,(\ell)} + \mathbf{R}_i$
($i = 0,\,\dotsc,\, N{-}1$) can be solved independently and hence in parallel.

A single, first sweep ($\ell = 0$) 
corresponds to solving the block-diagonal
system $\mathcal D\,\hat{\mathbf{Z}} = \mathbf{R}$.

Dropping generally the coupling $\mathcal U$ amounts to neglecting the backward
transport of adjoint information from future intervals.  
A fully dissipative pH system will ensure that this information decays exponentially, causing the successive iterates to converge rapidly.


\paragraph{Convergence analysis.}

The error in the Block-Jacobi iteration propagates backward through the
\emph{amplification matrix}
\begin{equation}\label{eq:Gamma-def}
  \Gamma_i := M_i^{-1}\, E_{11}
  = \bigl(E_{11} + k_i\, S^\top\bigr)^{-1}\, E_{11},
  \qquad i = 0,\dots,N{-}1.
\end{equation}

After $\ell$ sweeps, the Jacobi error will involve $\Gamma_i^{\,\ell}$. Thus, 
the iteration converges for arbitrary right-hand side terms
if and only if $\Gamma_i$ is a strict contraction.

\begin{proposition}
\label{prop:contraction}
Let Assumptions~\ref{ass:ph-structure} hold and assume assume that the smallest eigenvalue $\alpha$ of $\tilde S$ satisfies $\alpha > 0$.
Then for every $k_i > 0$ and every $v \in \R^r \setminus \{0\}$, $\Gamma_i$ is a
strict contraction in the $E_{11}$-norm, i.e., $\|\Gamma_i\, v\|_{E_{11}} < \|v\|_{E_{11}}$.
\end{proposition}

\begin{proof}

Set $w := \Gamma_i\,v$. Then
$M_i\, w = E_{11}\, v$, i.e.,
$E_{11}(v - w) = k_i\, S^\top w$.
Expanding $\|v\|_{E_{11}}^2 = \|w + (v{-}w)\|_{E_{11}}^2$ gives
\begin{equation*}
  \|v\|_{E_{11}}^2
  = \|w\|_{E_{11}}^2
  + 2\,\langle E_{11}(v{-}w),\, w \rangle
  + \|v{-}w\|_{E_{11}}^2.
\end{equation*}
Substituting $E_{11}(v{-}w) = k_i\,S^\top w$ and using $\langle S^\top w, w\rangle = \langle \widetilde S\, w, w\rangle
\geq \alpha\,\|w\|^2 > 0$
yields
\begin{equation*}
  \|v\|_{E_{11}}^2
  = \|w\|_{E_{11}}^2
  + 2k_i\,\langle \widetilde S\, w, w\rangle
  + \|v-w\|_{E_{11}}^2 > \| w\|_{E_{11}}^2 = \|\Gamma_i v \|_{E_{11}}^2 ,
\end{equation*}
which is the contraction. 
\end{proof}
We 
comment on the central assumption of the above result: {$\lambda_{\min}(\widetilde{S})>0$.}

In this case, the dynamics are exponentially stable such that perturbations are exponentially damped. 
{Thus, }
the sensitivities are strongly localized, allowing for decoupled solution of the adjoint. We refer to \cite{grune2022efficient} for a similar result in the context of goal-oriented error estimation in optimal control problems on large time horizons.

\begin{remark}[Quantitative contraction rate]\label{rem:contraction-rate}
The spectral radius $\rho(\Gamma_i)$ of the iteration matrix $\Gamma_i$
describes the contraction rate of the Jacobi iteration.
Since $\Gamma_i = (E_{11} + k_i S^\top)^{-1} E_{11}$, any eigenvalue
$\lambda$ of $\Gamma_i$ satisfies $E_{11}\,v = \lambda(E_{11} + k_i S^\top)v$
for some $v \neq 0$, i.e.\ $(1-\lambda)\,E_{11}\,v = \lambda\,k_i\,S^\top v$.
Taking the inner product with $v$ and using
$\langle S^\top v, v\rangle = \langle \widetilde S\, v, v\rangle
+ \mathrm{i}\,\mathrm{Im}\langle S^\top v, v\rangle$ gives
\begin{equation}\label{eq:spectral-radius}
  |\lambda| = \frac{1}{|1 + k_i\sigma|} \leq \frac{1}{1 + k_i\,\mu_{\min}},
  \qquad
  \mu_{\min} := \lambda_{\min}(\widetilde S,\, E_{11}),
\end{equation}
where $\sigma$ is the corresponding generalised eigenvalue of $(S^\top, E_{11})$,
which satisfies $\mathrm{Re}(\sigma) \geq \mu_{\min} > 0$.
Equality holds when $S$ is symmetric; in general, the imaginary part of $\sigma$
provides additional contraction beyond the bound~\eqref{eq:spectral-radius}.
The upper bound decreases monotonically as $k_i$ increases: larger time steps
yield stronger contraction.
Figure~\ref{fig:spectral} validates~\eqref{eq:spectral-radius} numerically.
\end{remark}

\begin{figure}[t]
  \centering
  \includegraphics[width=\linewidth]{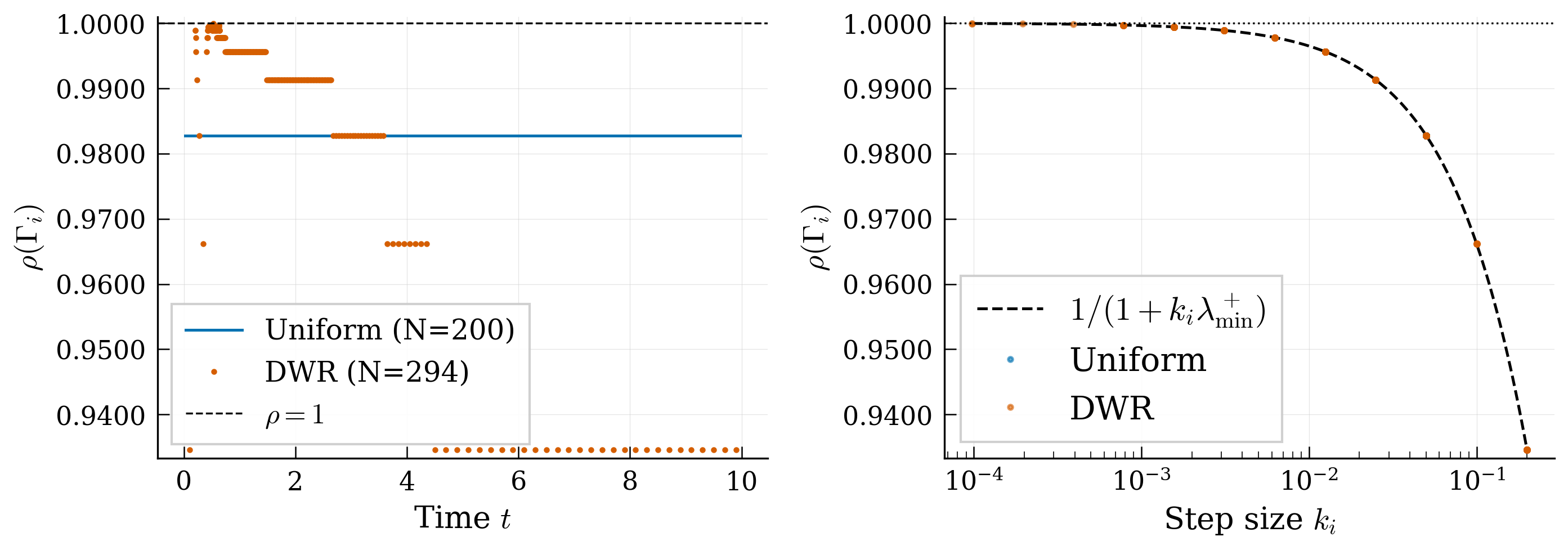}
  \caption{Spectral radius $\rho(\Gamma_i)$ of the amplification
    matrix~\eqref{eq:Gamma-def}, tested on \eqref{eq:tranmission-line-parameters-convergence}.
    \emph{Left:} distribution over $[0,T]$ on a uniform ($N=200$)
    and DWR-adaptive ($N=302$) grid; the adaptive grid places short
    intervals near $t=0$ where the input is active, resulting in
    $\rho(\Gamma_i)$ close to one in that region.
    \emph{Right:} $\rho(\Gamma_i)$ versus step size $k_i$;
    all values confirm the theoretical decay of
    $(1 + k_i\,\mu_{\min})^{-1}$.}
  \label{fig:spectral}
\end{figure}

\section{Numerical experiments}
\label{sec:numerical_experiments}

To numerically validate the error estimator and the efficiency of the adaptive
algorithm, we consider
two linear index-$1$
pH-DAEs: (a) an academic benchmark and (b) transmission line.  
The implementation uses the dG(0) discretization for both the
primal and adjoint problems and the algebraic error indicator derived in
Section~\ref{sec:DWR-analysis}. Furthermore, we demonstrate the effectivity of the parallel version.

At the end, we discuss and show numerically that the DWR-marked set stabilizes with rate that is significantly higher than the convergence rate of the adjoint

within the block Jacobi iteration.

\subsection{Academic example}

We consider the pH-DAE \eqref{eq:ph-dae} for $x(t)\in\mathbbm{R}^3$ with
\begin{equation}\label{eq:academic-example}
\begin{gathered}
    E = \text{diag}(1, 1, 0), \quad 
    R = \text{diag}(0.5,0.5,0.1), \quad
    Q = I_{3}, \quad 
    B = (1, 0, 0)^\top,
\\
  J = \begin{pmatrix} 0 & 1 & -1 \\ -1 & 0 & 0 \\ 1 & 0 & 0 \end{pmatrix}, \quad
    u(t) = \begin{cases} \sin(2\pi t) & \text{for } t < 0.5, \\ 0 & \text{for } t \geq 0.5. \end{cases}
\end{gathered}
\end{equation}
The simulation is performed on the time interval $I = [0, 1]$ with initial condition $x(0) = (1, 0, 0)^\top$.

The goal functional $\mathcal{J}$ on time grid $\mathcal{T}$ is the violation of the energy balance given by
(\ref{eq:J-def}--\ref{eq:Gi-splitting}). As reference solution, we use a uniformly grid with $N=20\,000$ time steps.
This solution yields a violation in $\mathcal{J}$ of approximately $1.25 \times 10^{-13}$.

\paragraph{Performance analysis of standard QoI.}

The focus is on three aspects of Algorithm~\ref{alg:adaptivity}:
(i)~convergence of the QoI,  
(ii)~decay of the global energy balance
violation, and (iii)~qualitative behavior of the adaptive time mesh
generated by the DWR estimator.

 Figure~\ref{fig:qoi_energy_convergence}~(a) shows the decay
of the QoI error as a function of the number of time steps.  Uniform
refinement converges at the expected rate.

DWR adaptivity achieves the same accuracy level with
roughly one order of magnitude fewer time steps across the full range
shown, indicating a significant reduction in the error constant.  This
behavior is explained by the adjoint weighting
in~\eqref{eq:simplified-indicator}: uniform refinement distributes
resolution evenly in time, whereas the DWR strategy selectively
refines those intervals that carry the largest contribution to the
leading-order error term in Theorem~\ref{thm:error-rep}, leaving
intervals with small adjoint-weighted residuals on a coarse mesh.

\begin{figure}[htb]
  \centering
  \begin{minipage}{0.48\textwidth}
    \centering
    \includegraphics[width=\linewidth]{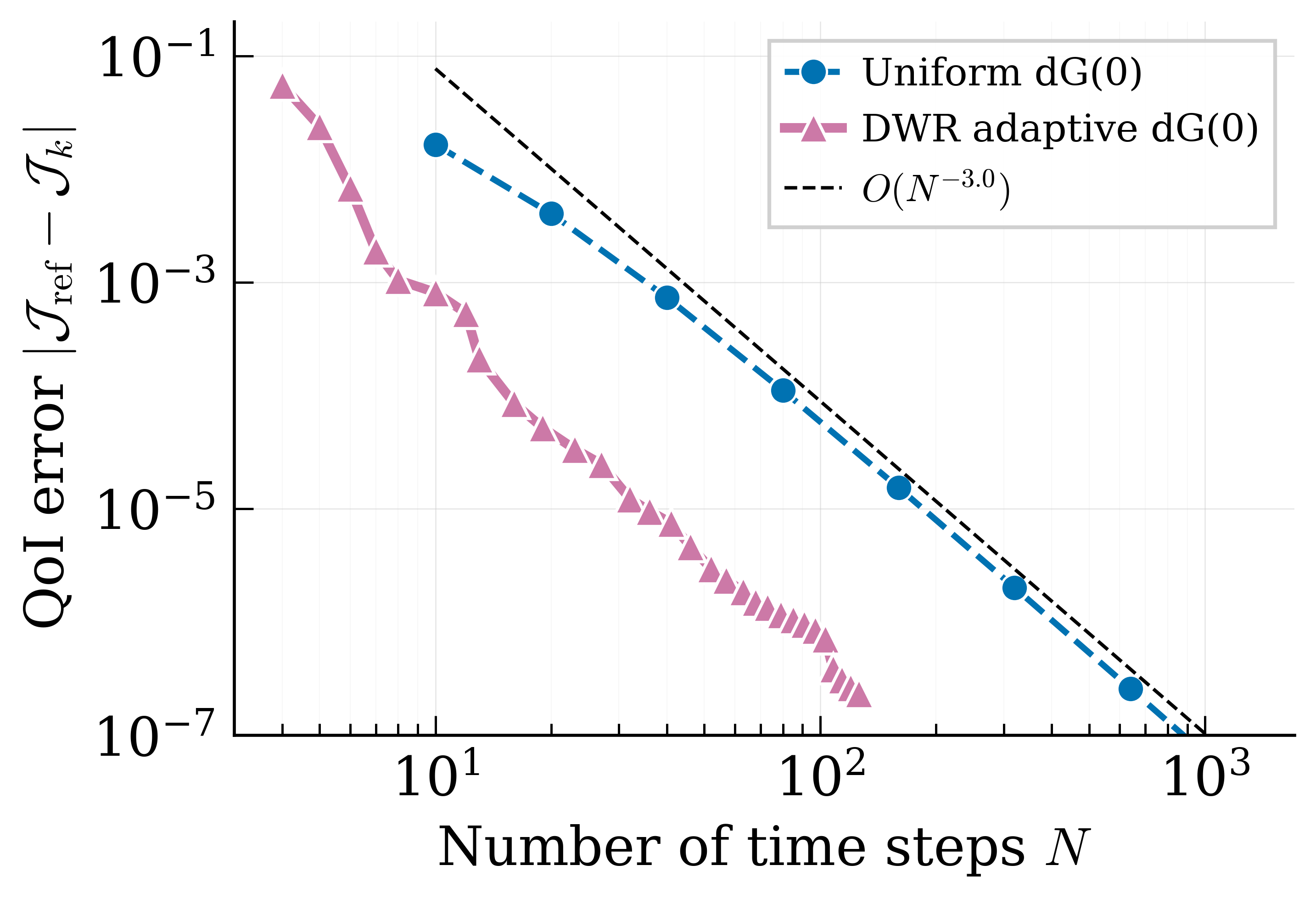}\\[2pt]
    (a) Decay of QoI error.
  \end{minipage}\hfill
  \begin{minipage}{0.48\textwidth}
    \centering
    \includegraphics[width=\linewidth]{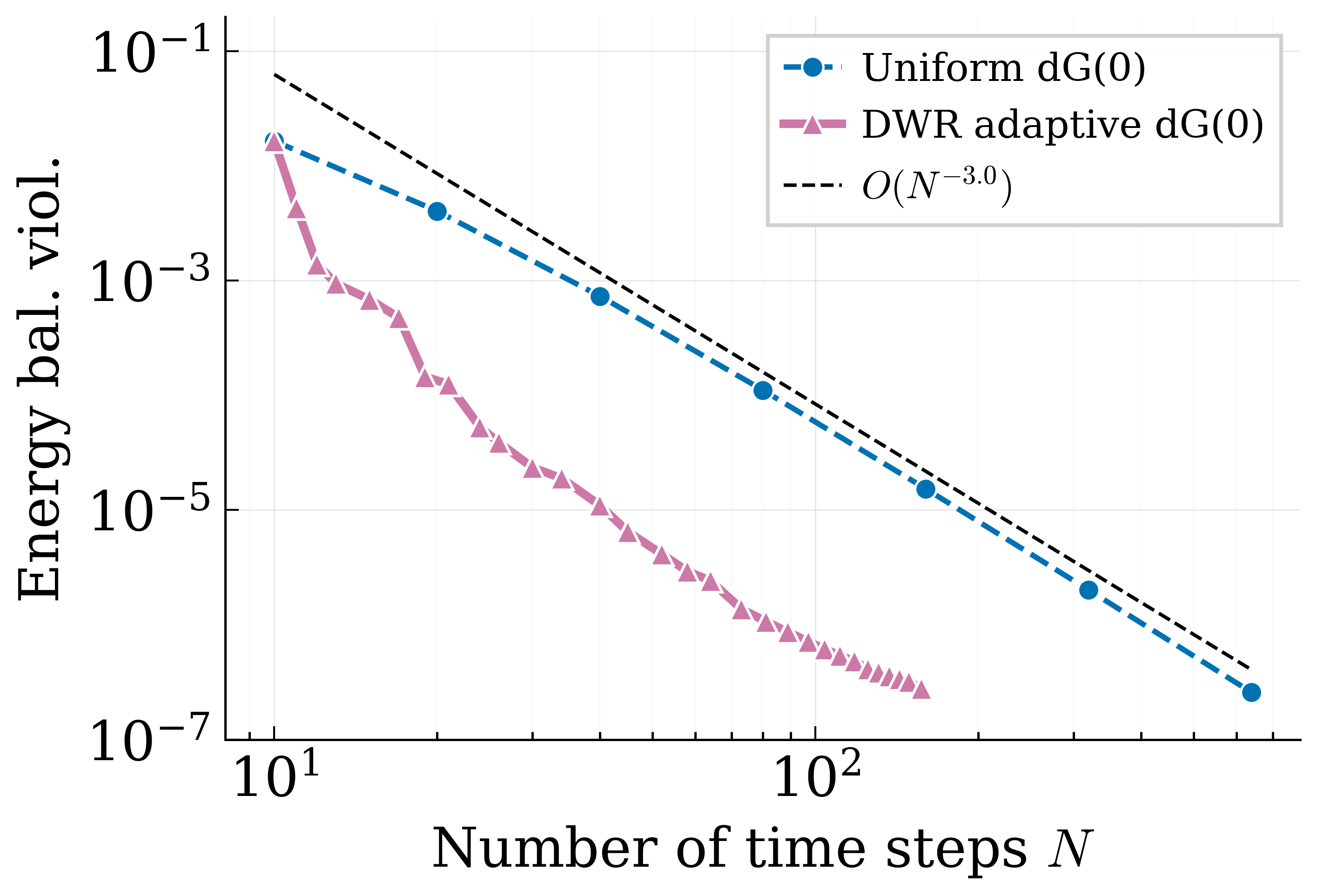}\\[2pt]
    (b) Decay of energy balance violation.
  \end{minipage}
  \caption{Convergence of uniform and DWR adaptive dG(0).
    (a)~Error in QoI 
    $|\mathcal J_{\mathrm{ref}} - \mathcal J_k|$ associated with the
    local energy balance ($J_{\mathrm{ref}}:=J (x_{\mathrm{ref}}$)).
    (b)~Global energy balance violation
    $\mathcal J(u_k) = \sum_i G_i^2$.
    The dashed line shows the 
    asymptotic rate $O(N^{-3})$,
    consistent with the dG(0) discretization.}
  \label{fig:qoi_energy_convergence}
\end{figure}

\medskip

\noindent \begin{minipage}[c]{0.58\textwidth}

\begin{remark}[Convergence rate of the QoI]\label{rem:qoi-rate}
The rate $\mathcal{J}(x_k) = O(k^3)$ follows from the structure of the
goal functional. By the a~priori estimate (Lemma~\ref{lem:apriori}),
$\|x_k - x_1\|_{L^2} = O(k)$, so each local residual satisfies
$G_i = O(k_i^2)$ (first-order state error integrated over an interval
of length~$k_i$). Since $\mathcal{J} = \sum_i G_i^2$, summing $N = T/k$
terms of size $O(k_i^4)$ gives $\mathcal{J} = O(k^3) = O(N^{-3})$.
\end{remark}

The adaptive refinement is illustrated in Figure~\ref{fig:mesh_evolution}. Starting from an initial uniform grid (iteration 0), the algorithm introduces additional time nodes predominantly near the initial time, where the system undergoes a rapid redistribution of energy from the prescribed initial condition.
In the course of refinements, a secondary time point of refinement emerges around $t = 0.45$. This clustering corresponds to the region where the energy residual $|G_i|$ and the adjoint sensitivity $\|z_i\|$ overlap most significantly. 
\end{minipage}
\hfill
\begin{minipage}[c]{0.39\textwidth}
\mbox{}\\[-\baselineskip]
  \centering
  \includegraphics[width=0.9\textwidth]{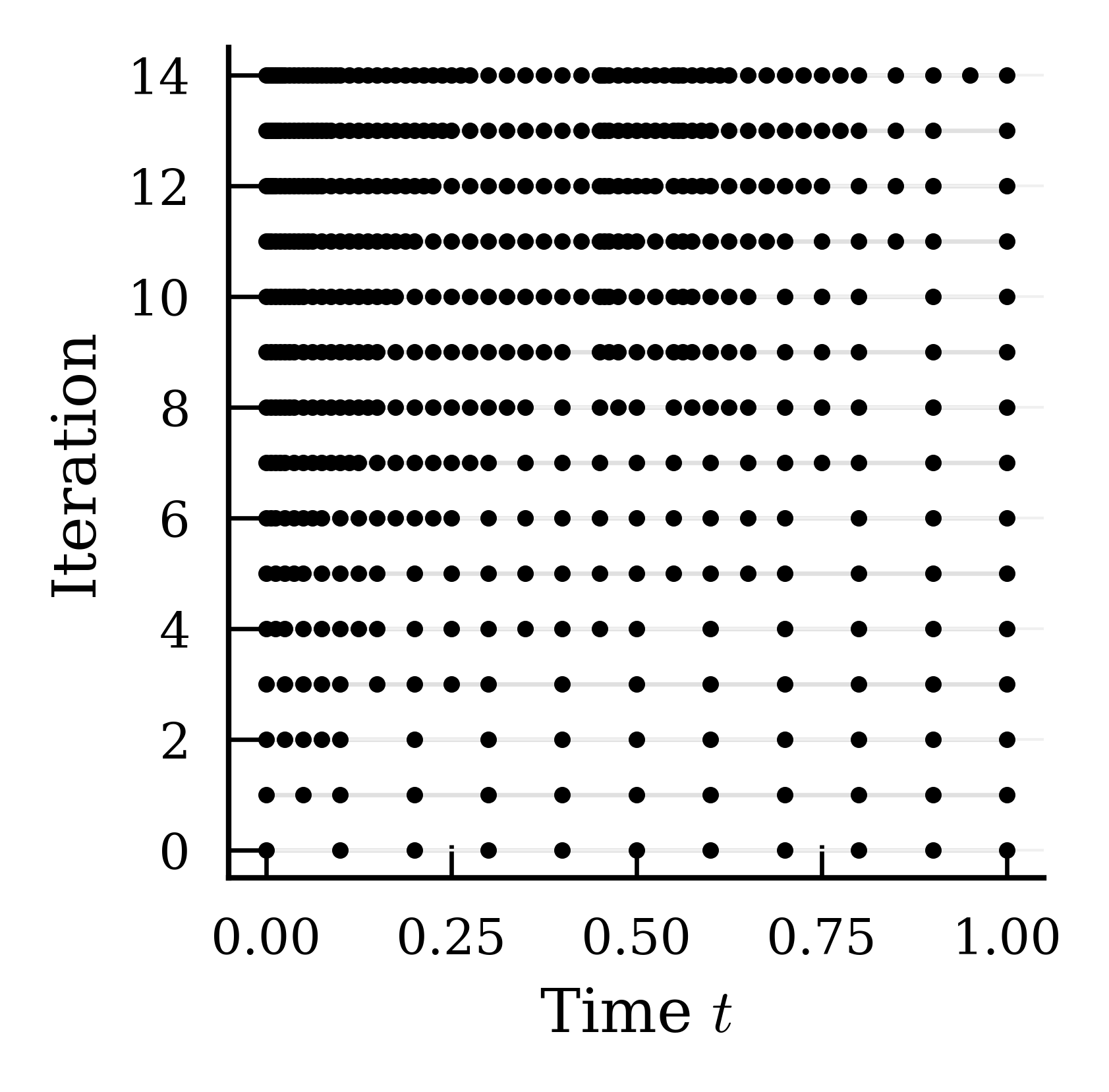}
  \\
  \refstepcounter{figure}  \label{fig:mesh_evolution}
  Figure~\ref{fig:mesh_evolution}: Evolution of time grid. 
\end{minipage}
\vspace{.05cm}

Importantly, the mesh remains relatively coarse for $t > 0.6$, confirming that once the system reaches a more dissipative state (cf. $u$ in \eqref{eq:academic-example}), the local energy balance is already well resolved without the need for additional degrees of freedom. 

 Figure~\ref{fig:adjoint_structure} illustrates the structure of the local error indicator. The energy residual $|G_i|$ is largest during the initial transient and exhibits a secondary peak near $t = 0.5$. Conversely, the adjoint sensitivity $\|z_i\|$ is characterized by a gradual decay from the initial time. Their product, shown in the right panel, exhibits a distinct profile that peaks near $t = 0$ and $t = 0.5$. This confirms that the adaptive algorithm refines precisely those intervals where the discretization error and the sensitivity of the goal functional are simultaneously large.

\begin{figure}[htb]
  \centering
  \includegraphics[width=0.95\textwidth]{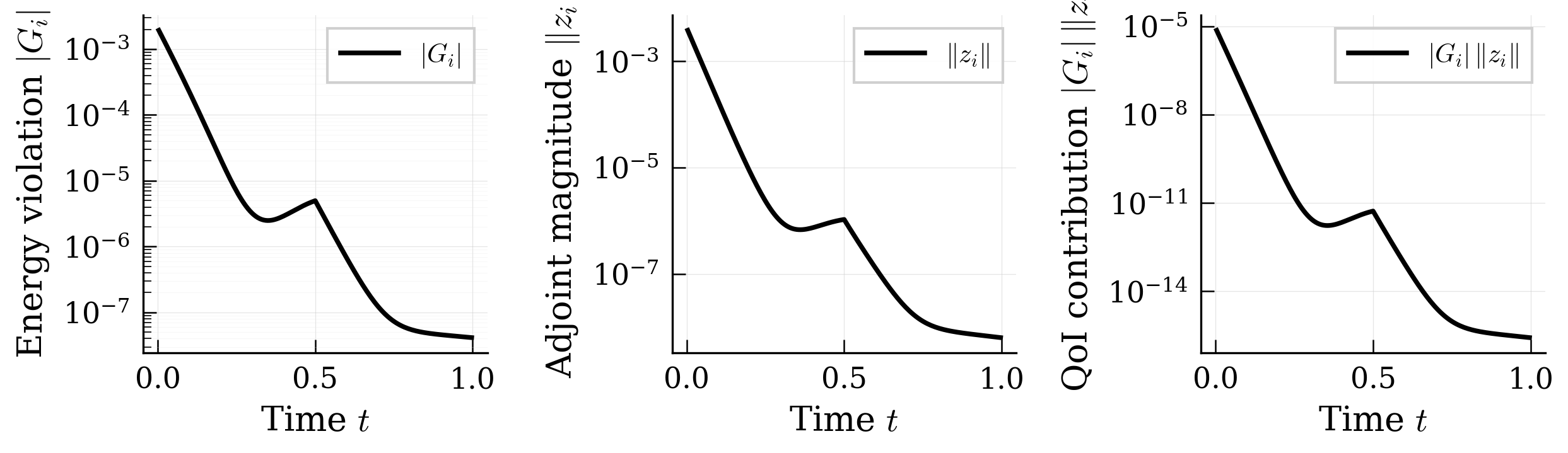}
  \caption{Structure of the DWR error indicator. Left: local energy residual magnitude $|G_i|$. Middle: adjoint sensitivity norm $\|z_i\}$. Right: weighted contribution $|G_i|\,\|z_i\|$, which drives adaptive refinement.}
  \label{fig:adjoint_structure}
\end{figure}

\subsection{Transmission line}

\label{sec:transmission-line}

We consider an RCL ladder transmission line from the port-Hamiltonian
benchmark collection~\cite{MoserBenchmark2021}, consisting of $n_s$
cascaded building blocks; see Fig.~\ref{fig:transmission_line}.

Using modified nodal analysis (MNA)~\cite{Freund2011}, the unknowns are the node voltages $e = (e_0,\dotsc,e_{2n_s})^\top$, the currents through inductances $\imath_L = (\imath_{L_1},\dotsc,\imath_{L_{n_s}})^\top$, and the current through the voltage-source $\imath_V$.

Using Kirchhoff's current law  at every node and adding branch equations for the inductance and voltage source, we obtain for the $k$th building block: ($k=1,\, \dotsc,\, n_s$)
\begin{alignat*}{2}
  \text{node $2k\!-\!1$:\quad}  && 0 & = \tfrac{1}{R} (e_{2k-2} -e_{2k-1}) - \imath_{L,k} \\
  \text{inductance:\quad}              &&L \dot{\imath}_{L,k} & = e_{2k-1} - e_{2k} \\
 \text{node $2k$:\quad}      &&C \dot{e}_{2k} &=  \imath_{L,k} - \tfrac{1}{R}(e_{2k}-e_{2k+1}) .
\end{alignat*}

Using branch related incidence matrices $\mathcal{A}_C,\mathcal{A}_L,\mathcal{A}_R,\mathcal{A}_V$ for capacitances, inductances, resistances
and voltage-source, 
we obtain a port-Hamiltonian model 
 of the benchmark circuit (Fig.~\ref{fig:transmission_line})
\begin{equation*}
  E\dot{x}(t) = (J - \mathcal{R})\,x(t) + B\,u(t), \qquad
  y(t) = B^\top x(t),
\end{equation*}
with unknown $x^\top(t) = (e^\top(t),\, \imath_L^\top(t),\, \imath_V^\top(t)) \in \R^{3n_s +2} $ and
\begin{equation}\label{eq:rcl-matrices}
  E = \begin{bmatrix} \mathcal{A}_C C \mathcal{A}_C^\top & 0 & 0 \\
                      0 & L & 0 \\ 0 & 0 & 0 \end{bmatrix}, \quad
  J = \begin{bmatrix} 0 & -\mathcal{A}_L & -\mathcal{A}_v \\
                      \mathcal{A}_L^\top & 0 & 0 \\
                      \mathcal{A}_V^\top & 0 & 0 \end{bmatrix}, \quad
  \mathcal{R} = \begin{bmatrix}
                      \mathcal{A}_R G \mathcal{A}_R^\top & 0 & 0 \\
                      0 & 0 & 0 \\ 0 & 0 & 0
                \end{bmatrix},
  \quad 
  B= \begin{bmatrix} 0 \\ 0 \\ 1\end{bmatrix}
\end{equation}
and matrices $C=\text{diag}(C_1,\dotsc, C_{n_s-1})$, 
$L=\text{diag}(L_1,\dotsc,L_{n_s})$, $G=\text{diag}(\frac{1}{R_0},\frac{1}{R_1},\dotsc,\frac{1}{R_{n_s+1}})$, see~\cite{MoserBenchmark2021}.

In fact, this is an index-1 system with algebraic equations for
nodes 0, $2k-1$ $(k=1,\, \dotsc,\, n_s)$, $2n_s$, and the branch equation for the voltage source.

Thus, the reduction with the Schur complement gives $r_{\mathrm{diff}}=2n_s -1$ differential equations. In our experiments, we set
\begin{equation*}
\begin{gathered}
  C_k = 1,\quad L_k = 1,
  \quad x(0) = (1,0,\dotsc,0)^\top, 
  \quad
  u(t) = 50\exp\bigl(-(t-0.5)^2/(2\cdot 0.05^2)\bigr). 
  \end{gathered}
  \end{equation*}
For the waveform study, we furthermore use
\begin{equation}\label{eq:transmission-line-parameters-waveform}
  n_s = 50, \quad R_k = 2, \quad [0,T] = [0,10],
\end{equation}
while for the convergence and adaptivity study we set
\begin{equation}\label{eq:tranmission-line-parameters-convergence}
n_s = 100, \quad R = 0.35,\quad 
\mathcal{A}_R G \mathcal{A}_R^\top + I_{n_s+1}
\quad [0,T]=[0,10].
\end{equation}

Adaptive refinement uses QoI~\eqref{eq:Iloc-global}, Dörfler parameter
$\theta = 0.5$, and an initial uniform grid of $N = 50$ intervals.
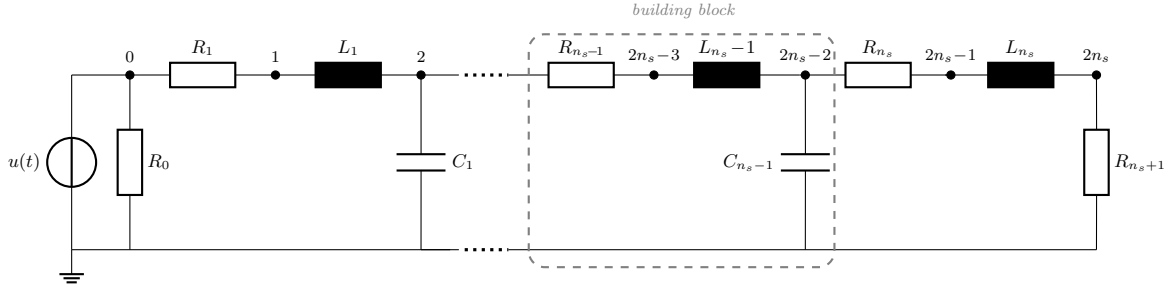
\begin{figure}[htbp]
  \centering
  \begin{circuitikz}[scale=0.77, transform shape,
      every node/.style={font=\small}]

    \draw (0.5,0) -- (7,0);

    \draw (0.5,0) node[ground]{};

    \draw (0.5,0) to[V, l=$u(t)$, invert] (0.5,3);
    \draw (0.5,3) -- (1.5,3);

    \draw[fill=black] (1.5,3) circle (0.07);
    \node[above=3pt] at (1.5,3) {\footnotesize $0$};

    \draw (1.5,3) to[R, l=$R_0$] (1.5,0);

    \draw (1.5,3) to[R, l=$R_1$] (4.0,3);
    \draw[fill=black] (4.0,3) circle (0.07);
    \node[above=3pt] at (4.0,3) {\footnotesize $1$}; 
    \draw (4.0,3) to[L, l=$L_1$] (6.5,3);

    \draw[fill=black] (6.5,3) circle (0.07);
    \node[above=3pt] at (6.5,3) {\footnotesize $2$}; 
    \draw (6.5,3) to[C, l=$C_1$] (6.5,0);

    \draw (6.5,3) -- (7.15,3);
    \draw (6.5,0) -- (7.15,0);

    \draw[dotted, very thick] (7.25,3) -- (8.,3);
    \draw[dotted, very thick] (7.25,0) -- (8.,0);

    \draw (8.,3) to[R, l=$R_{n_s\!-\!1}$] (10.5,3);
    \draw[fill=black] (10.5,3) circle (0.07);
    \node[above=3pt] at (10.5,3) {\footnotesize $2n_s\!-\!3$}; 
    \draw (10.5,3) to[L, l=$L_{n_s}\!-\!1$] (13.,3);

    \draw (13,3)--(13.1,3);

    \draw[fill=black] (13.1,3) circle (0.07);
    \node[above=3pt] at (13.1,3) {\footnotesize $2n_s\!-\!2$}; 
    \draw (13.1,3) to[C, l_=$C_{n_s - 1}$] (13.1,0);

    \draw (13.1,3) to[R, l=$R_{n_s}$] (15.6,3);
    \draw[fill=black] (15.6,3) circle (0.07);
    \node[above=3pt] at (15.6,3) {\footnotesize $2n_s\!-\!1$}; 
    \draw (15.5,3) to[L, l=$L_{n_s}$] (18.1,3);
    \draw[fill=black] (18.1,3) circle (0.07);
    \node[above=3pt] at (18.1,3) {\footnotesize $2n_s$}; 

    \draw (18.1,3) to[R, l=$R_{n_s+1}$] (18.1,0);

    \draw (8.0,0) -- (18.1,0);

    \draw[dashed, rounded corners=6pt, thick, gray]
        (8.35, -0.3) rectangle (13.6, 3.7);
    \node[gray, font=\footnotesize\itshape] at (11., 4.1)
        {building block};
\end{circuitikz}
  \caption{RCL ladder transmission line~\cite{MoserBenchmark2021,MoserPHDAE2023} with building block given in dashed box.
    The right end is terminated by load resistor~$R_{n_s+1}$ instead of a capacitance.}
  \label{fig:transmission_line}
\end{figure}

\begin{remark}[Coercivity of $\widetilde{S}$]%
\label{rem:coercivity}
We discuss the regularization $\mathcal{A}_R G \mathcal{A}_R^\top + 1 I_{n_s+1}$ in \eqref{eq:tranmission-line-parameters-convergence}.  

In fact, the Schur complement splits the dissipation into
two contributions. 
First, the inductor-currents are damped by the Schur term $J_{12}A_{22}^{-1}J_{21}$ (series resistor to the inducances)

Second, the node voltages have no resistive path to ground. Thus, they do not experience dissipation. This is attacked by the regularization term.

Physically, the regularization models a 
leakage conductance 
$\varepsilon=1$ at each capacitor.
\end{remark}

\paragraph{Physical behavior.}
Figure~\ref{fig:waveform} shows the node voltage $v_k(t)$ at six
representative nodes for the ideal benchmark ($n_s = 50$.

From node~0 the input (voltage source signal) propagates along the ladder:

successive nodes reach
their peak later and with smaller amplitude, reflecting the finite
propagation speed and cumulative resistive losses in the series
elements.
By node~26 the signal is heavily attenuated, with amplitude an order of
magnitude smaller than at node~1.
This pattern---localised transient activity concentrated in an early
time window that shifts rightward along the chain---is precisely the
structure that the DWR estimator detects and exploits for selective
temporal refinement.

\paragraph{Convergence and cost-to-target.}
Figure~\ref{fig:convergence} shows the energy-balance violation
$\sum_i G_i^2$ versus number of intervals $N$ for the dissipative
benchmark \eqref{eq:tranmission-line-parameters-convergence}.
Both uniform and DWR-adaptive dG(0) achieve the rate $O(N^{-3.0})$
in the asymptotic regime, consistent with the second-order primal
discretisation and the quadratic QoI.
The DWR curve lies substantially below the uniform curve throughout,
reflecting a reduced error constant: the estimator correctly identifies
the short window near $t = 0$ where the Gaussian pulse drives rapid
variation at node~1 as the primary source of the global energy defect,
concentrating refinement there while leaving the slowly varying
dissipative tail on a coarse grid.

Table~\ref{tab:cost_to_target} quantifies the interval savings:
We reach up to $89\%$ (for targets from $10^{-1}$ to $10^{-2}$).

The gain is superlinear in the target: as the tolerance tightens,
uniform refinement must reduce every interval to compensate for the
globally distributed error, whereas Dörfler marking continues to
concentrate work at 
intervals with larger adjoint-weighted
residuals.
In the asymptotic regime, both methods converge at the same rate, so
the DWR advantage is a reduced leading error constant.

\begin{figure}[t]
  \centering
  \begin{subfigure}[t]{0.48\linewidth}
    \centering
    \includegraphics[width=\textwidth,keepaspectratio]{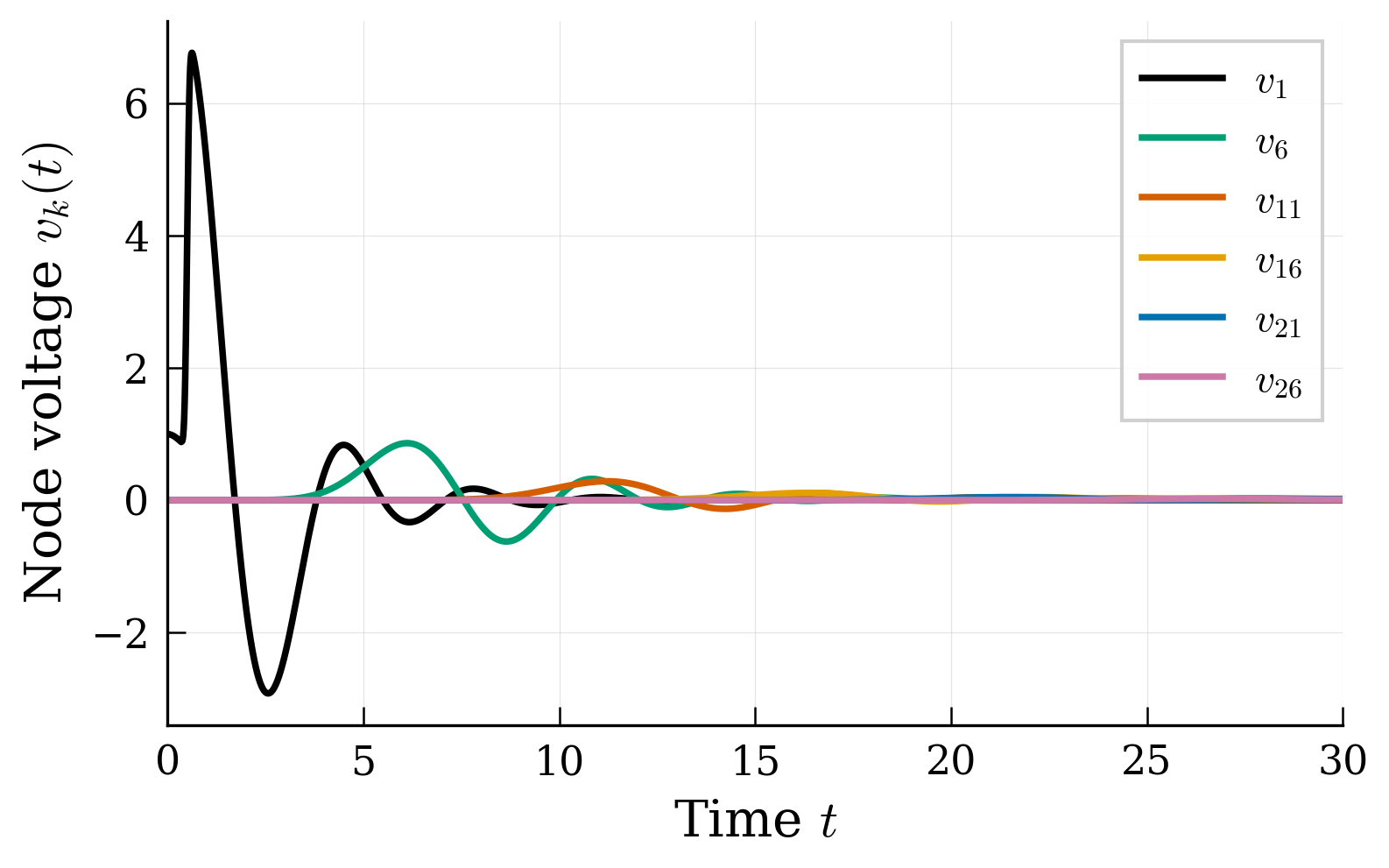}
    \caption{Node voltage $v_k(t)$ at representative nodes
             (ideal benchmark, $n_s = 50$).}
    \label{fig:waveform}
  \end{subfigure}
  \hfill
  \begin{subfigure}[t]{0.48\linewidth}
    \centering
    \includegraphics[width=\textwidth,keepaspectratio]{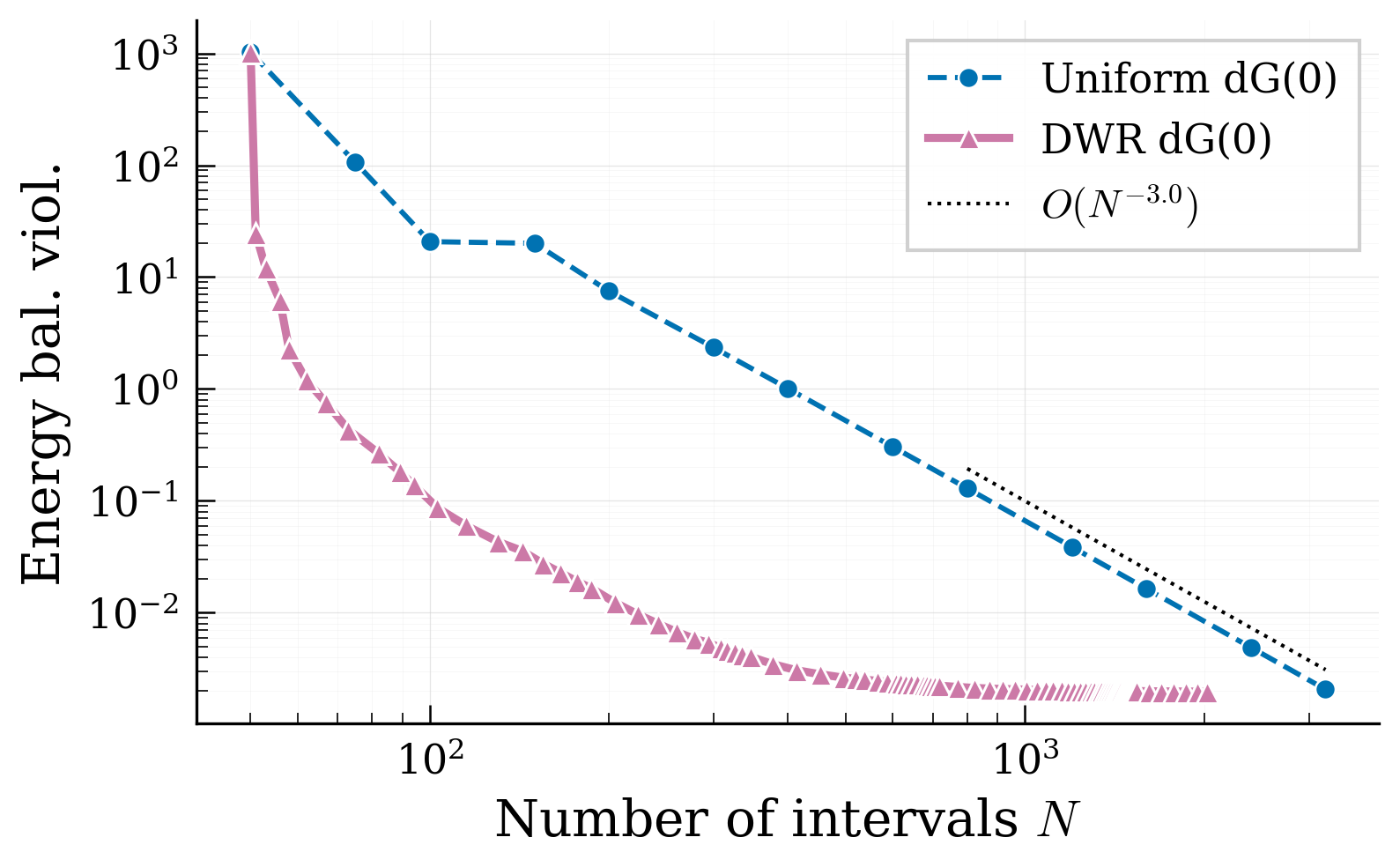}
    \caption{Energy-balance violation vs.\ intervals
             (dissipative benchmark, $n_s = 100$).}
    \label{fig:convergence}
  \end{subfigure}
  \caption{RCL ladder transmission line.
    \emph{Left}: node voltages at nodes 1, 6, 11, 16, 21, 26 on the
    final adaptive grid; the excitation propagates from left to right
    with increasing delay and decay of amplitude.
    \emph{Right}: energy-balance violation $\sum_i G_i^2$ under
    uniform~($\bullet$) and DWR-adaptive~($\blacktriangle$) dG(0)
    refinement; the dotted line gives the fitted rate
    $O(N^{-3.0})$.}
  \label{fig:waveform_convergence}
\end{figure}

\begin{table}[t]
  \centering
  \caption{Intervals required to reduce $\sum_i G_i^2$ below a
    prescribed target.
    Savings $= 1 - N_{\mathrm{DWR}}/N_{\mathrm{Uniform}}$.
    Dissipative benchmark with parameters~(\ref{eq:tranmission-line-parameters-convergence})
    .}
  \label{tab:cost_to_target}
  \begin{tabular}{c@{\qquad}r@{\qquad}r@{\qquad}c}
    \toprule
    Target & $N$ (Uniform) & $N$ (DWR) & Savings \\
    \midrule
    $10^{2}$  &    76\; &   51\; & 33\% \\
    $10^{1}$  &   184\; &   54\; & 71\% \\
    $10^{0}$  &   401\; &   64\; & 84\% \\
    $10^{-1}$ &   872\; &   101\; & 88\% \\
    $10^{-2}$ &  1887\; &  206\; & 89\% \\
    \bottomrule
  \end{tabular}
\end{table}


\begin{figure}[htbp]
  \centering
  \begin{minipage}{0.48\textwidth}
    \centering
    \includegraphics[width=\textwidth]{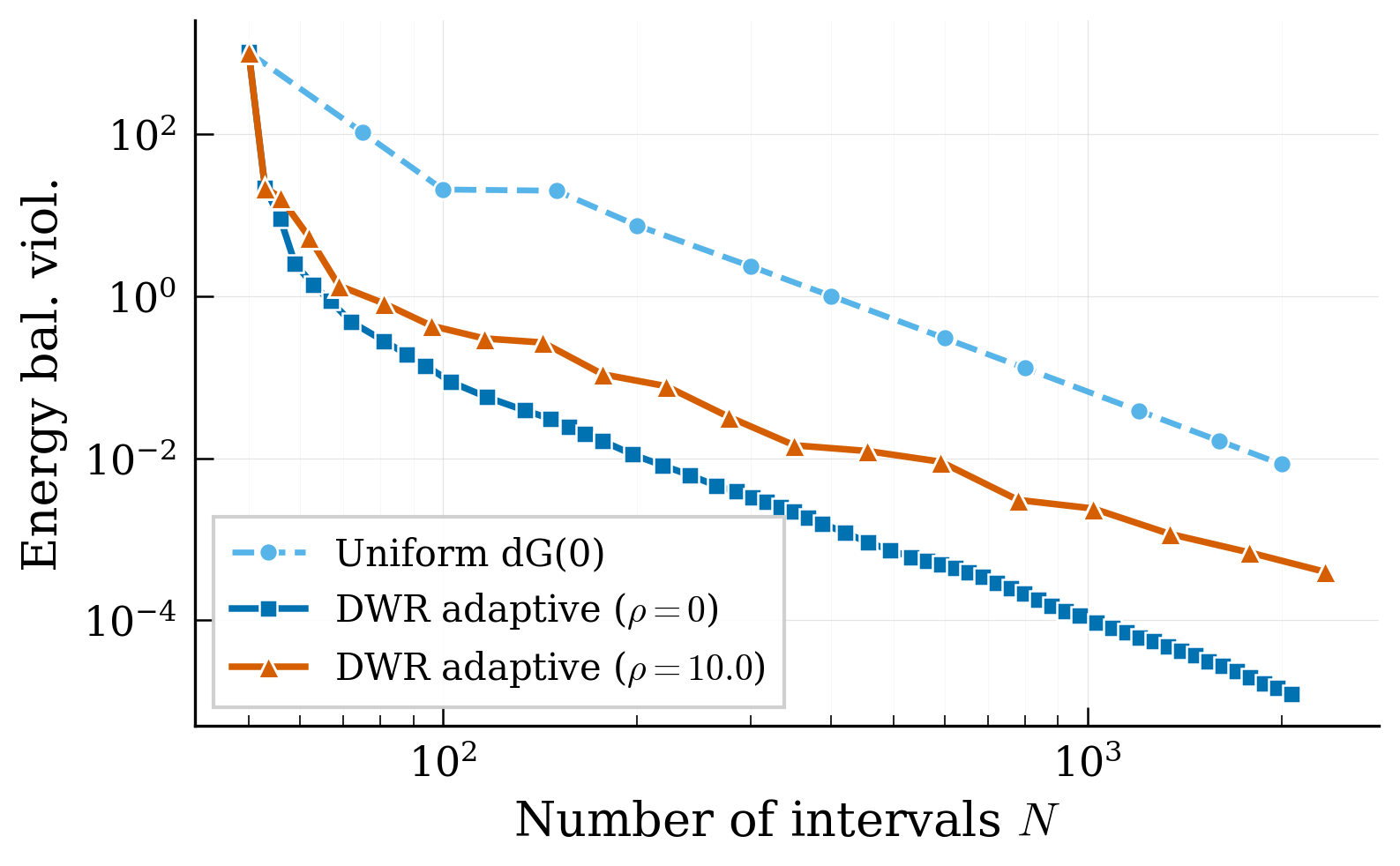}\\[4pt]
    (a) Energy-balance violation.
  \end{minipage}\hfill
  \begin{minipage}{0.48\textwidth}
    \centering
    \includegraphics[width=\textwidth]{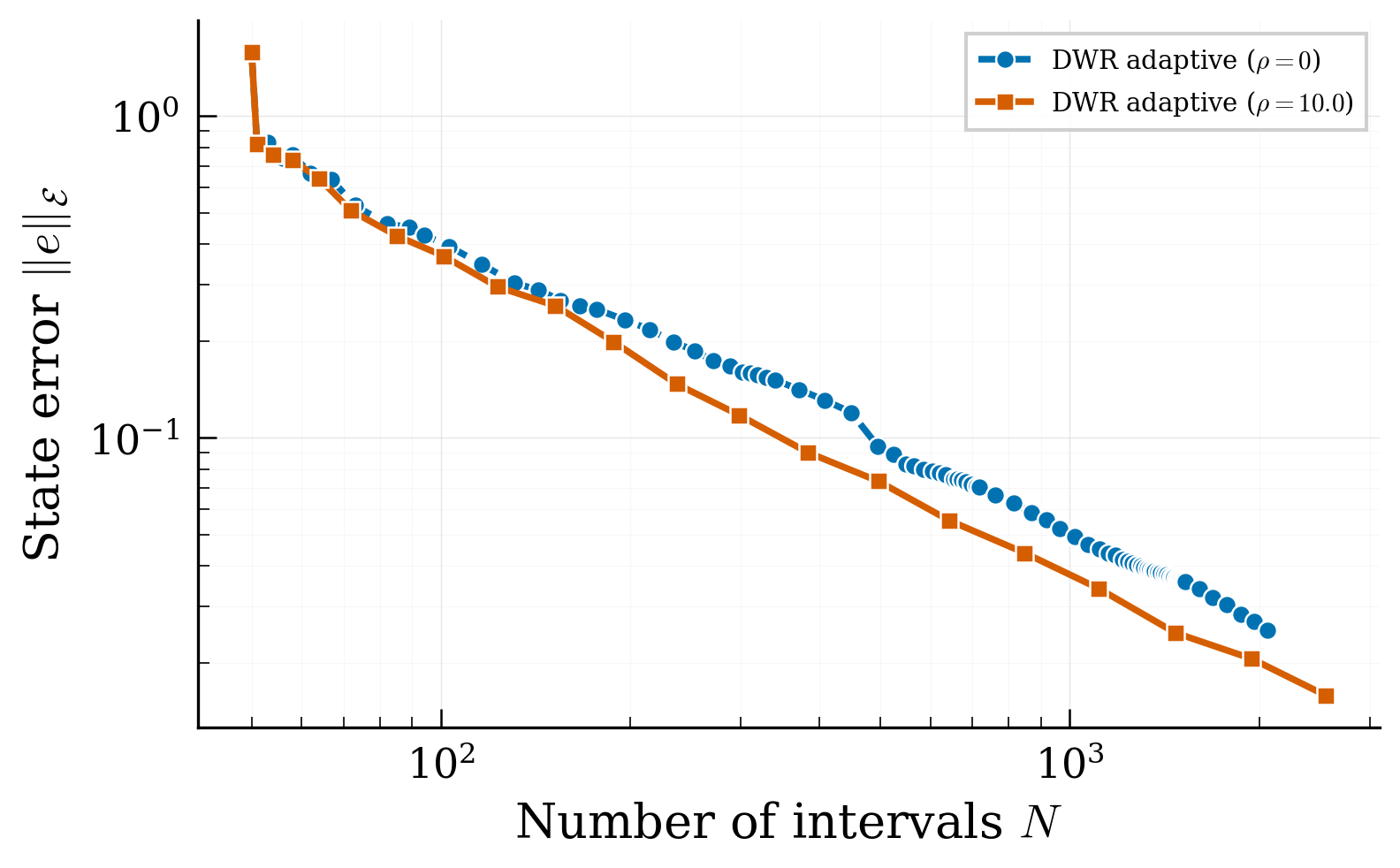}\\[4pt]
    (b) State error in the energy norm.
  \end{minipage}
  \caption{Effect of the Hamiltonian-weighted QoI on the transmission-line
    benchmark~\eqref{eq:tranmission-line-parameters-convergence} ($\rho=10$).
    (a)~Energy-balance violation $\sum_i G_i^2$ for uniform dG(0) and
    both adaptive strategies.
    (b)~State error $\|e\|_{\mathcal{E}}$ vs.\ number of intervals $N$.
    The weighted strategy trades a mild increase in energy-balance
    residual for a consistent reduction in state error throughout the
    full $N$-range.}
  \label{fig:weighted_qoi_norm}
\end{figure}

\paragraph{Norm-augmented QoI.}
We now consider the augmented QoI~\eqref{eq:weighted_qoi_norm} applied
to the regulairzed transmission-line benchmark~\eqref{eq:tranmission-line-parameters-convergence}.
The Hamiltonian integral is discretised as $\rho\sum_i k_i H(x_{i+1})$,
introducing an additional distributed source term $\rho\,k_i\,\nabla H(x_{i+1})$
in the adjoint right-hand side at every interval.
The adaptive strategy therefore balances two objectives simultaneously:
reduction of the local energy-balance residual and control of the state
energy throughout the trajectory.

Figure~\ref{fig:weighted_qoi_norm}~(a) shows that both adaptive strategies
substantially outperform uniform refinement.
The weighted functional ($\rho=10$) incurs a moderate penalty in the
energy-balance component relative to the 
pure energy case 
($\rho=0$):
for large $N$, the weighted curve lies roughly half an order of magnitude above the other case.
This is expected, since the DWR estimator redistributes temporal resolution to also improve the state errors.

The benefit is visible in Figure~\ref{fig:weighted_qoi_norm}~(b).
The two strategies track each other closely for $N \le 200$;

afterwards
the weighted curve pulls consistently below the other 
one, achieving roughly a factor of two improvement in state error at
$N = 2000$.
This trade-off, a mild degradation in energy-balance accuracy in
exchange for a uniform improvement in state accuracy, is the
expected consequence of augmenting the QoI with the Hamiltonian term.

\paragraph{Effectivity of DWR estimator.}
We assess the estimator from Section~\ref{sec:computable-estimator} on
the regularized transmission line model \eqref{eq:tranmission-line-parameters-convergence}. 

The respective Schur complement $S$ \eqref{eq:S-and-F-def} satisfies
$\|S - S^\top\|/\|S\| \approx 1.93$, reflecting the dominant
skew-symmetric (conservative) component of the interconnection.
A fine reference solution with $N = 50\,000$ uniform steps provides
$J_{\mathrm{ref}} = 5.434368 \times 10^{-7}$.
The adaptive loop starts from~$N = 49$ and proceeds via
Algorithm~\ref{alg:adaptivity} with D\"orfler parameter $\theta = 0.5$.

Table~\ref{tab:effectivity} records the effectivity index
\begin{equation}\label{eq:effectivity}
  I_{\mathrm{eff}}
  := \frac{\bigl|\sum_i \eta_i\bigr|}
          {\bigl|J_{\mathrm{ref}} - J(u_k)\bigr|}
\end{equation}
at selected refinement levels.

\begin{table}[htb]
  \centering
  \caption{DWR effectivity indices for the transmission line pH-DAE (for selected iterations).}
  \label{tab:effectivity}
  \begin{tabular}[t]{@{}rrccr@{}}
    \toprule
    $\ell$ & $N$ & $|J_{\mathrm{ref}} - J_k|$
    & $\sum \eta_i$ & $I_{\mathrm{eff}}$ \\
    \midrule
     0 &  49 & $1.0013\cdot 10^{+3}$ & $-5.5261\cdot 10^{+2}$ & 0.552 \\
     1 &  50 & $3.2829\cdot 10^{+1}$ & $-4.6589\cdot 10^{+1}$ & 1.419 \\
     3 &  54 & $6.7644\cdot 10^{+0}$ & $-2.9892\cdot 10^{+1}$ & 4.419 \\
     5 &  58 & $2.6419\cdot 10^{+0}$ & $-3.7683\cdot 10^{+0}$ & 1.426 \\
     7 &  66 & $8.4843\cdot 10^{-1}$ & $-2.2573\cdot 10^{+0}$ & 2.661 \\
     9 &  81 & $2.5157\cdot 10^{-1}$ & $-8.4141\cdot 10^{-1}$ & 3.345 \\
    11 &  94 & $1.3149\cdot 10^{-1}$ & $-2.8819\cdot 10^{-1}$ & 2.192 \\
    \bottomrule
  \end{tabular}\quad\;\; 
  \begin{tabular}[t]{@{}rrccr@{}}
    \toprule
    $\ell$ & $N$ & $|J_{\mathrm{ref}} - J_k|$
    & $\sum \eta_i$ & $I_{\mathrm{eff}}$ \\
    \midrule
    13 & 113 & $6.6317\cdot 10^{-2}$ & $-2.3767\cdot 10^{-1}$ & 3.584 \\
    15 & 142 & $3.6818\cdot 10^{-2}$ & $-1.8662\cdot 10^{-1}$ & 5.069 \\
    17 & 161 & $2.4414\cdot 10^{-2}$ & $-9.9402\cdot 10^{-2}$ & 4.071 \\
    19 & 178 & $1.5597\cdot 10^{-2}$ & $-7.2259\cdot 10^{-2}$ & 4.633 \\
    21 & 217 & $8.6004\cdot 10^{-3}$ & $-3.4843\cdot 10^{-2}$ & 4.051 \\
    23 & 258 & $5.1810\cdot 10^{-3}$ & $-2.3090\cdot 10^{-2}$ & 4.457 \\
    24 & 275 & $4.2982\cdot 10^{-3}$ & $-1.6962\cdot 10^{-2}$ & 3.946 \\
    \bottomrule
  \end{tabular}
\end{table}

\paragraph{Bounded effectivity.}
After the initial transient ($\ell = 0,1$), the effectivity index
remains bounded and exhibits moderate oscillations, predominantly in the range $[2,\,5]$ with an average around~$3.5$--$4$.
While occasional peaks (e.g., $I_{\mathrm{eff}}\approx 5.1$ at $\ell=15$) occur, the index does not display any systematic growth or deterioration under refinement.

Since $\mathcal{J} = \sum_i G_i^2$ is quadratic in the $G_i$ and each
$G_i$ is itself quadratic in~$x_1$, the cubic remainder in the error
identity~\eqref{eq:DWR-identity} is of the same order as the leading
term. Consequently, asymptotic exactness $I_{\mathrm{eff}} \to 1$ is not expected; instead, a bounded effectivity index is the relevant benchmark~\cite{Rannacher2003,
EndtmayerLangerWick2020}.

\paragraph{One-sided overestimation.}
Both the true error $J_{\mathrm{ref}} - J_k$ and the estimator
$\sum \eta_i$ are consistently negative, indicating that the
estimator overestimates the magnitude of the error.  This is the
conservative direction: the stopping criterion $\eta_{\mathrm{tot}}
\leq \mathit{TOL}$ will not terminate prematurely.

\subsection{Early stabilization of the marked set within the Jacobi iteration}
\label{sec:early-stabilization}

Proposition~\ref{prop:contraction} guarantees global convergence of the
Jacobi iteration, but the adaptive algorithm requires only that the
D\"orfler-marked set stabilize.  We now show numerically that this occurs
well before adjoint convergence.

We denote by~$\mathcal{M}_{\mathrm{ex}}$ the D\"orfler-marked set
(cf.\ Algorithm~\ref{alg:adaptivity}, step~10) obtained from the exact sequential
adjoint~$z_k$, and by~$\mathcal{M}^{(\ell)}$ the marked set produced by
the $\ell$-th Jacobi sweep.  We say the marked set has \emph{stabilized}
at sweep~$k^*$ if
$\mathcal{M}^{(\ell)} = \mathcal{M}_{\mathrm{ex}}$ for
$\ell = k^*, \dots, k^* + 3$ i.e., agreement persists over four consecutive sweeps. Early stabilization is explained by a complementarity
property of the DWR indicator: the quantity
$\eta_i^{(\ell)} = |\langle \mathcal{R}_i,\,\Delta z_i^{(\ell)}\rangle|$
couples the primal residual~$\mathcal{R}_i$ to the adjoint
difference~$\Delta z_i^{(\ell)}$.  In the pH-DAE setting, $\mathcal{R}_i$
is large only on the few intervals where the primal solution is
under-resolved and negligible elsewhere.  On intervals with
$\mathcal{R}_i \approx 0$ the indicator vanishes regardless of the
pointwise adjoint error, so Jacobi error confined to the low-residual tail
does not corrupt the marking.

Table~\ref{tab:jacobi} and Figure~\ref{fig:J24} provide numerical evidence
on the transmission-line benchmark~\eqref{eq:tranmission-line-parameters-convergence}.  On the
initial uniform grid ($N = 50$, $|\mathcal{M}_{\mathrm{ex}}| = 1$), a
single sweep suffices ($k^* = 1$), yielding a $50\times$ speedup: the
uniform step size produces strong contraction
($\rho \approx 0.93$) and the sole marked interval is resolved
immediately by the local right-hand side.  On the late adaptive grid
($N = 143$, $|\mathcal{M}_{\mathrm{ex}}| = 12$, after 14~DWR
refinements), $k^* = 110$ sweeps are needed and the speedup reduces
to~$1\times$.  This degradation is driven by two effects: adaptive
refinement concentrates short intervals near~$t = 0$, where the local
contraction rates satisfy $\rho(\Gamma_i) \approx 1$; and the propagation
distance from the terminal boundary to the furthest marked interval
spans nearly all $143$~intervals.

Figure~\ref{fig:J4} confirms the complementarity mechanism:
the Jacobi indicators match the exact curve 
in the marked region (left
of the $\theta$-cut) well before~$k^*$; deviations are confined to the
negligible tail.

\begin{remark}[Scaling of the sweep count]
  The ratio $k^*/|\mathcal{M}_{\mathrm{ex}}|$ grows from~$1$ to~$9.2$
  across refinement levels (Table~\ref{tab:jacobi}), so $k^*$ is not
  governed by~$|\mathcal{M}_{\mathrm{ex}}|$ alone.  Rather, it depends on
  the propagation distance~$d$ (number of intervals between the furthest
  marked interval and the terminal boundary) and the local contraction
  rates~$\rho(\Gamma_i) = (1 + k_i\,\mu_{\min})^{-1}$, which degrade on
  short intervals.  The practical benefit of the Block-Jacobi
  approximation is thus concentrated in the early and intermediate
  refinement iterations ($\ell \leq 10$), where step-size ratios are
  moderate and speedups of $2$--$6\times$ are realized.
\end{remark}

\begin{table}[htbp]
  \centering
  \caption{Jacobi sweep convergence across DWR refinement levels on the
    transmission-line benchmark
    \eqref{eq:tranmission-line-parameters-convergence}.
    $k^*$: first sweep at which
    $\mathcal{M}^{(\ell)} = \mathcal{M}_{\mathrm{ex}}$.
    Speedup~$= N/k^*$.}
  \label{tab:jacobi}
  \begin{tabular}{rrrrrr|@{\quad}rrrrrr}
    \toprule
    DWR & & & & & & DWR\\
    iter~$\ell$ & $N$ & $|\mathcal{M}_{\mathrm{ex}}|$
      & $k^*$ & Speedup & $\rho_{\mathrm{worst}}$ 
     & 
    iter~$\ell$ & $N$ & $|\mathcal{M}_{\mathrm{ex}}|$
      & $k^*$ & Speedup & $\rho_{\mathrm{worst}}$ \\
    \midrule
     0 &  50 &  1 &   1 & $50\times$ & 0.935 & 
     8 &  82 &  7 &  40 &  $2\times$ & 0.998 \\
     2 &  53 &  3 &   9 &  $6\times$ & 0.983 & 
    10 &  94 &  9 &  40 &  $2\times$ & 0.999 \\
     4 &  58 &  4 &  10 &  $6\times$ & 0.996 & 
    12 & 116 & 15 &  75 &  $2\times$ & 1.000 \\
     6 &  67 &  6 &  25 &  $3\times$ & 0.996 & 
    14 & 143 & 12 & 110 &  $1\times$ & 1.000 \\
    \bottomrule
  \end{tabular}
\end{table}

\begin{figure}[htbp]
  \centering
  \begin{subfigure}[b]{0.49\linewidth}
    \centering
    \includegraphics[width=\linewidth]{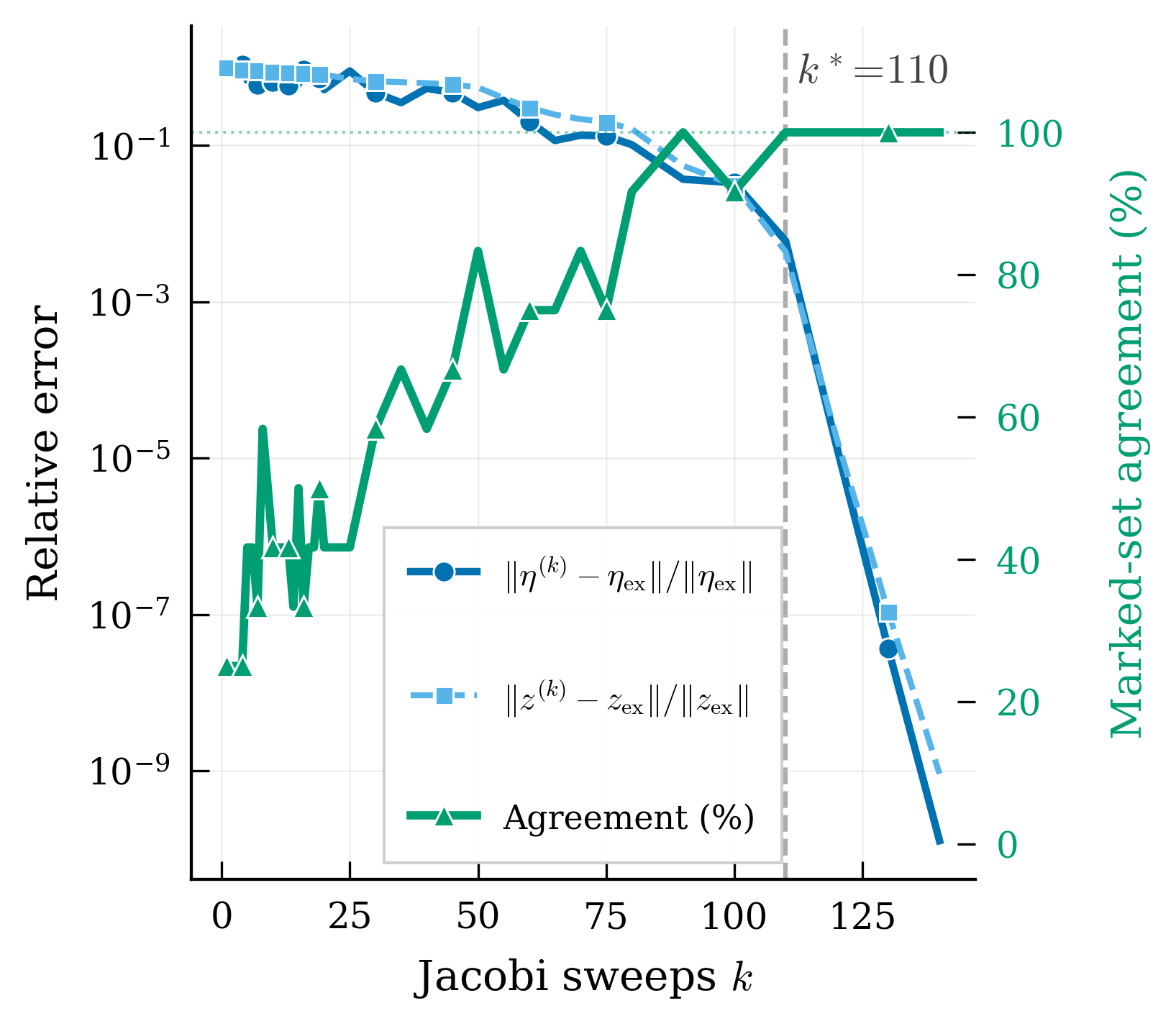}
    \caption{Error and agreement vs.\ sweeps~$k$.}
    \label{fig:J2}
  \end{subfigure}
  \hfill
  \begin{subfigure}[b]{0.49\linewidth}
    \centering
    \includegraphics[width=\linewidth]{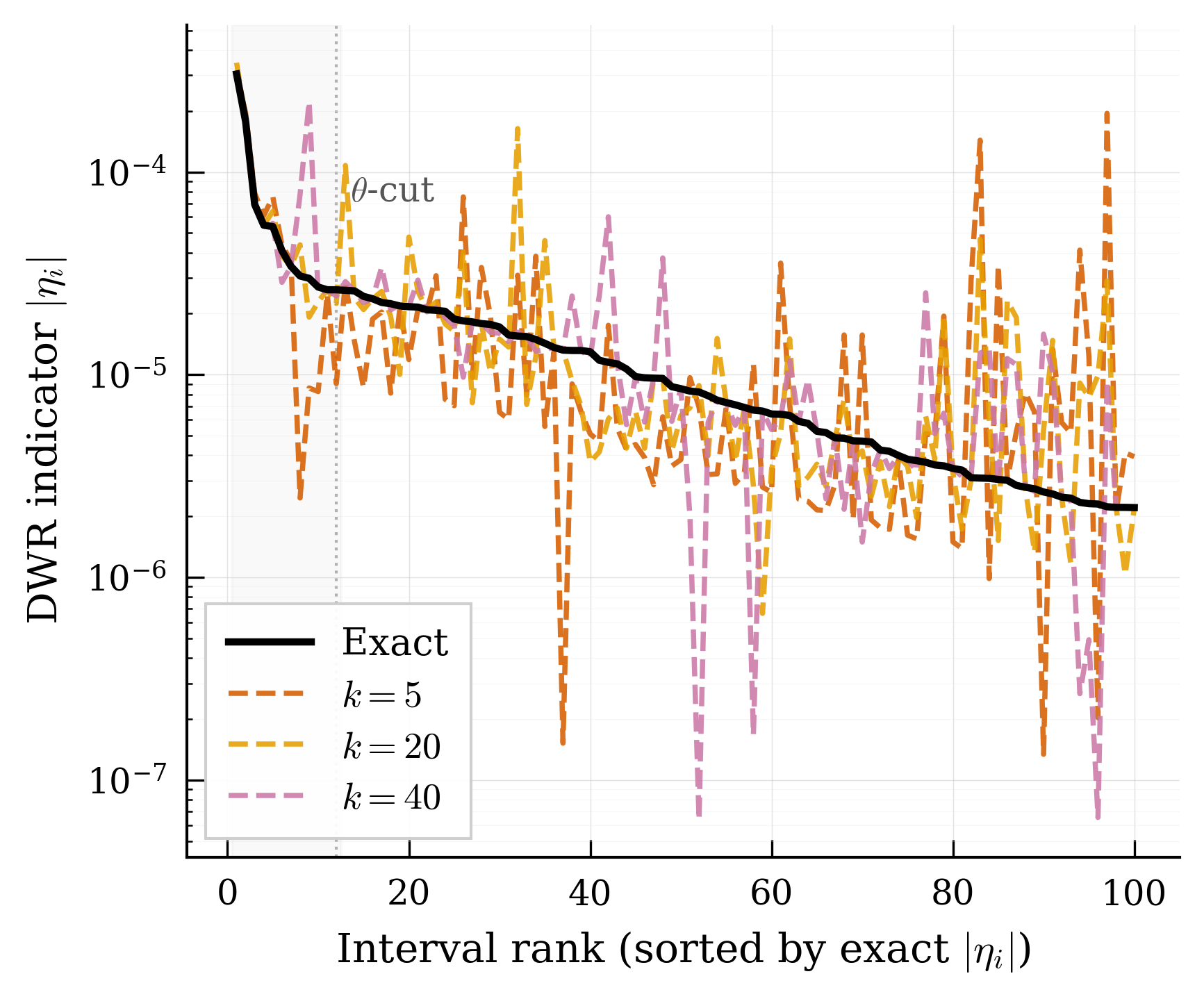}
    \caption{Ranked indicators; shaded:
      $|\mathcal{M}_{\mathrm{ex}}| = 12$.}
    \label{fig:J4}
  \end{subfigure}
  \caption{Late adaptive grid ($N = 143$,
    $|\mathcal{M}_{\mathrm{ex}}| = 12$, 14~DWR refinements).
    (\subref{fig:J2})~Relative indicator error (solid) and adjoint error
    (dashed), with marked-set agreement (right axis, green).
    Stabilization at $k^* = 110$.
    (\subref{fig:J4})~DWR indicators ranked by~$|\eta_i|$ for
    $k \in \{5, 20, 40\}$.  Jacobi approximations match the exact curve
    within the marked region well before~$k^*$.}
  \label{fig:J24}
\end{figure}

\section{Conclusion}
\label{sec:conclusions}

We have presented a goal-oriented time-adaptive scheme for linear port-Hamiltonian DAEs of index~1, combining the dual weighted residual method with the intrinsic energy-balance structure of pH systems. The approach treats the violation of the energy balance as a computable quantity of interest and adapts the time grid to reduce it, with full freedom in the choice of the underlying time integrator. The main contributions are: (i)~a Schur-complement reduction that formulates the DWR framework entirely in the differential variable; (ii)~a complete a~posteriori error representation with computable local indicators; (iii)~a parallelizable Block-Jacobi approximation of the adjoint whose contraction is guaranteed by the dissipative pH structure; and (iv)~numerical validation demonstrating savings of up to 89\% in the number of time intervals compared to uniform refinement.

For our specific quantity of interest 
$\mathcal{J}(u_k) = \sum_i G_i^2$,
the adjoint equation is driven by the right-hand side
$2G_i\,\nabla_{x_{i+1}} G_i$, such that the dual solution $z_i$ is
structurally correlated with the primal residuals $G_i$ themselves.
As a consequence, any local marking criterion proportional to $|G_i|$
identifies essentially the same dominant intervals as the global DWR
estimator $\eta_i = \langle z_i, r_i\rangle$.  The advantage of DWR
over such heuristics therefore lies not primarily in interval
selection but in the rigorous \emph{a posteriori} error representation
$\mathcal{J}(u) - \mathcal{J}(u_k) \approx \sum_i \eta_i$ and the
resulting effectivity index (Table~\ref{tab:effectivity}), which
provides certified error control unavailable to purely local methods.
For goal functionals that are not themselves residual-based---such as
a pointwise output evaluation $\mathcal{J}(u_k) = |y(t^*)|^2$
for a fixed observation time~$t^*$---the
adjoint genuinely identifies intervals where local error is small but
global sensitivity is large, and the advantage of DWR over local
strategies would be substantially larger.

The dG(0) discretization used in this work is first-order accurate, and the DWR theory extends directly to higher-order discontinuous Galerkin methods~\cite{meidner2008}. Replacing dG(0) by dG($s$) for $s \geq 1$ would yield higher-order convergence rates in both the state and the goal functional, with the adjoint computation and indicator evaluation following the same block structure derived in Sections~\ref{sec:DWR-analysis} and~\ref{sec:efficient_adjoint}. A second natural direction is the extension to nonlinear port-Hamiltonian systems, where the Hamiltonian or the structure matrices depend on the state. The Lagrangian-based error identity (Theorem~\ref{thm:error-rep}) does not require linearity of the underlying dynamics, but the stability analysis and the contraction guarantee for the Block-Jacobi iteration (Proposition~\ref{prop:contraction}) rely on the linear structure and would require appropriate generalizations. These extensions are the subject of ongoing work.
\appendix

\section{Proof of Lemma~\ref{lem:scaledtounscaled}}
\label{sec:app}
\begin{proof}
For clarity of exposition, we present the proof for a constant step size
$k_i \equiv k > 0$; the extension to variable steps is immediate, since
the contraction $P_i := (E_{11} + k_i S)^{-1} E_{11}$ satisfies
$\|P_i\| \leq 1$ uniformly in $k_i \geq 0$ (as $E_{11} + k_i S \ge E_{11}$
in the $E_{11}$-inner product), so all norm estimates in what follows hold
with $P$ replaced by $P_i$ at each step and the products
$P^{i-j+1}$ replaced by $P_i P_{i-1} \cdots P_{j+1}$, each of which is
again a contraction.

Now, let us assume $\mu >0$.
Denote $M:=E_{11}+kS$ and $M_\mu:=E_{11}+k(S+\mu E_{11})=M+k\mu E_{11}$.
From the recursions for $x_k$ and $(x_\mu)_k$ (with the same source and initial data), we have
\begin{align*}
M x_k^i - M_\mu (x_\mu)_k^i
=
E_{11}\bigl(x_k^{i-1}-(x_\mu)_k^{i-1}\bigr) 
\end{align*}
and hence
\begin{align*}
     M\bigl(x_k^i-(x_\mu)_k^i\bigr)
=
E_{11}\bigl(x_k^{i-1}-(x_\mu)_k^{i-1}\bigr)+k\mu E_{11}(x_\mu)_k^i.
\end{align*}
Consequently, it holds
\begin{align*}
x_k^i-(x_\mu)_k^i
=
M^{-1}E_{11}\bigl(x_k^{i-1}-(x_\mu)_k^{i-1}\bigr)
+
k\mu\,M^{-1}E_{11}(x_\mu)_k^i.
\end{align*}
We set $e^i:=x_k^i-(x_\mu)_k^i$ and abbreviate $P:=M^{-1}E_{11}$ which, due to dissipativity of $S$, is a contraction: $M = E_{11} + kS \ge E_{11}$, that is, $\|P\|\leq 1$. Since $x_k^0 = (x_\mu)_k^0 = x_{1,0}$, we have $e^0 = 0$. The recursion
\begin{align*}
    e^i=P\, e^{i-1}+k\mu\, P\, (x_\mu)_k^i, \qquad i = 1, \dots, N,
\end{align*}
upon repeated substitution and using $e^0 = 0$, yields the error representation
\begin{align}\label{eq:errorrep}
 e^i = k\mu \sum_{j=1}^i P^{\,i-j+1}(x_\mu)_k^j.
\end{align}

Throughout the proof, we use the parametric Young inequality
\begin{align}\label{eq:paramYoung}
  \|a + b\|^2 \leq (1+\varepsilon)\|a\|^2 + (1+\varepsilon^{-1})\|b\|^2,
  \qquad \varepsilon > 0.
\end{align}

\medskip\noindent
\emph{Step~1 (Weighted norm $\sum k\|x_k^i\|^2$).}
From~\eqref{eq:errorrep} and $\|P^{i-j+1}\| \leq 1$, we have
\begin{align}\label{eq:ei-bound}
  \|e^i\|^2
  \leq (k\mu)^2 \Bigl(\sum_{j=1}^{i}\|(x_\mu)_k^j\|\Bigr)^{\!2}
  \leq (k\mu)^2\, i \sum_{j=1}^{i}\|(x_\mu)_k^j\|^2
\end{align}
using Young's inequality for $i$ terms. 
Applying~\eqref{eq:paramYoung} to $x_k^i = (x_\mu)_k^i + e^i$ and summing gives
\begin{align}\label{eq:p2}
  \sum_{i=1}^{N} k\,\|x_k^i\|^2
  \leq (1+\varepsilon)\sum_{i=1}^{N} k\,\|(x_\mu)_k^i\|^2
     + (1+\varepsilon^{-1})\sum_{i=1}^{N} k\,\|e^i\|^2.
\end{align}
For the error sum, we substitute~\eqref{eq:ei-bound} and exchange the order of summation:
\begin{align*}
  \sum_{i=1}^{N} k\,\|e^i\|^2
  &\leq k(k\mu)^2 \sum_{i=1}^{N} i \sum_{j=1}^{i} \|(x_\mu)_k^j\|^2
  = k(k\mu)^2 \sum_{j=1}^{N} \|(x_\mu)_k^j\|^2 \sum_{i=j}^{N} i \\
  &\leq k(k\mu)^2 \cdot \frac{N^2}{2}\sum_{j=1}^{N}\|(x_\mu)_k^j\|^2
  = \frac{\mu^2 T^2}{2}\sum_{j=1}^{N} k\,\|(x_\mu)_k^j\|^2
\end{align*}
with $k N = T$.
Substituting this into~\eqref{eq:p2} yields
\begin{align}\label{eq:weightednorm}
  \sum_{i=1}^{N} k\,\|x_k^i\|^2
  \leq \Bigl(1 + \varepsilon + (1+\varepsilon^{-1})\frac{\mu^2 T^2}{2}\Bigr)
       \sum_{i=1}^{N} k\,\|(x_\mu)_k^i\|^2.
\end{align}

\medskip\noindent
\emph{Step~2 (Terminal value $\|x_k^{N}\|^2$).}
From~\eqref{eq:ei-bound} with $i = N$, we find
\begin{align*}
  \|e^{N}\|^2
  \leq (k\mu)^2 N \sum_{j=1}^{N}\|(x_\mu)_k^j\|^2
  = \mu^2 T \sum_{j=1}^{N} k\,\|(x_\mu)_k^j\|^2.
\end{align*}
Applying~\eqref{eq:paramYoung} gives
\begin{align}\label{eq:terminal}
  \|x_k^{N}\|^2
  \leq (1+\varepsilon)\|(x_\mu)_k^{N}\|^2
     + (1+\varepsilon^{-1})\mu^2 T \sum_{j=1}^{N} k\,\|(x_\mu)_k^j\|^2.
\end{align}

\medskip\noindent
\emph{Step~3 (Jump term $\sum\|x_k^{i+1} - x_k^i\|^2$).}
Define $d^i := e^{i+1} - e^i$ and $g^i := (x_\mu)_k^{i+1} - (x_\mu)_k^i$.
From the recursion $e^{i+1} = P\,e^i + k\mu\, P\,(x_\mu)_k^{i+1}$
and $e^i = P\,e^{i-1} + k\mu\, P\,(x_\mu)_k^i$, subtraction gives
\begin{align}\label{eq:d-recursion}
  d^i = P\, d^{i-1} + k\mu\, P\, g^i, \qquad i = 1, \dots, N{-}1,
\end{align}
with $d^0 = e^1 - e^0 = e^1 = k\mu\, P\,(x_\mu)_k^1$ (using $e^0 = 0$).
Since $\|P\| \leq 1$, taking norms in~\eqref{eq:d-recursion} and telescoping yields
\begin{align*}
  \|d^i\| \leq \|d^0\| + k\mu \sum_{l=1}^{i}\|g^l\|
           = k\mu\Bigl(\|(x_\mu)_k^1\| + \sum_{l=1}^{i}\|g^l\|\Bigr).
\end{align*}
Squaring via Young's inequality and summing, we find
\begin{align}\label{eq:sumd}
  \sum_{i=1}^{N-1}\|d^i\|^2
  &\leq 2(k\mu)^2 \Bigl((N{-}1)\|(x_\mu)_k^1\|^2
        + \sum_{i=1}^{N-1}i\sum_{l=1}^{i}\|g^l\|^2\Bigr) \nonumber\\
  &\leq 2(k\mu)^2 N\,\|(x_\mu)_k^1\|^2
        + (k\mu)^2 N^2 \sum_{l=1}^{N-1}\|g^l\|^2 \nonumber\\
  &= 2k T\mu^2\,\|(x_\mu)_k^1\|^2
     + \mu^2 T^2 \sum_{l=1}^{N-1}\|g^l\|^2.
\end{align}
For the first term, $k\|(x_\mu)_k^1\|^2 \leq \sum_{i=1}^{N} k\,\|(x_\mu)_k^i\|^2$, so
\begin{align}\label{eq:sumd-final}
  \sum_{i=1}^{N-1}\|d^i\|^2
  \leq 2T\mu^2 \sum_{i=1}^{N} k\,\|(x_\mu)_k^i\|^2
     + \mu^2 T^2 \sum_{l=1}^{N-1}\|g^l\|^2.
\end{align}
Applying~\eqref{eq:paramYoung} to $x_k^{i+1} - x_k^i = g^i + d^i$, summing over $i = 1, \dots, N{-}1$, and substituting~\eqref{eq:sumd-final}:
\begin{align}\label{eq:jumpbound-inner}
  \sum_{i=1}^{N-1}\|x_k^{i+1} - x_k^i\|^2
  &\leq (1 + \varepsilon + (1{+}\varepsilon^{-1})\mu^2 T^2)
        \sum_{i=1}^{N-1}\|g^i\|^2
  + (1{+}\varepsilon^{-1})\,2T\mu^2
        \sum_{i=1}^{N} k\,\|(x_\mu)_k^i\|^2.
\end{align}

\medskip\noindent
\emph{Step~4 (Assembly).}
Combining~\eqref{eq:weightednorm}, \eqref{eq:terminal}, and~\eqref{eq:jumpbound-inner}, each term on the left-hand side of~\eqref{eq:integralestimate} is bounded by a constant times the corresponding terms on the right-hand side.
Choosing $\varepsilon = \mu T$ and collecting, the overall constant satisfies
\begin{align*}
  c(\mu,T)
  = 1+ \mu T + \max \left\{ \left( 1 \!+\! \tfrac{1}{\mu T}\right) \tfrac{\mu^2 T^2}{2},\, \left(1 \!+\! \tfrac{1}{\mu T} \right) 2  \mu^2 T \right\} 
  = 1+\mu T + \mu\max\left\{ \mu T +  \tfrac{1}{2}\mu T^2,\;  2\mu T +2 \right\}\!.
\end{align*}
Hence, $c(\mu,T) \xrightarrow{\mu \to 0} 1$. Finally, the estimate trivially holds for $\mu=0$.
\end{proof}

\bibliographystyle{plainnat}
\bibliography{references}
\end{document}